\documentclass[12pt]{article}

\RequirePackage{etex}
\RequirePackage{amsmath,amsfonts,amssymb}
\RequirePackage{mathtools}
\RequirePackage{graphicx}
\RequirePackage{pstricks}
\RequirePackage{epic,eepic}
\RequirePackage{color}
\RequirePackage{multicol}
\RequirePackage{longtable}
\RequirePackage{ifthen}
\RequirePackage{calc}
\RequirePackage{tikz-cd}
\RequirePackage{enumerate}

\DeclareMathOperator*{\Rtimes}{\rtimes}

\DeclareMathOperator{\Ker}{Ker}

\let\im\relax

\DeclareMathOperator{\im}{im}

\DeclareMathOperator{\id}{id}
\DeclareMathOperator{\Aut}{Aut}
\DeclareMathOperator{\Out}{Out}
\DeclareMathOperator{\Inn}{Inn}

\newcommand{\gen}[1]{\left\langle#1\right\rangle} 
\newcommand{\then}{\Rightarrow}

\renewcommand{\iff}{\Leftrightarrow}
\newcommand{\Iff}{\Longleftrightarrow}

\newcommand{\zz}{\mathbb{Z}}

\newcommand{\rr}{\mathbb{R}}

\newtheorem{lemma}{Lemma}[section]
\newtheorem{theorem}[lemma]{Theorem}
\newtheorem{proposition}[lemma]{Proposition}
\newtheorem{corollary}[lemma]{Corollary}
\newtheorem{definition}[lemma]{Definition}
\newenvironment{proof}{\medskip\noindent\textit{Proof: \/}}
               {\begin{flushright}\rule{2mm}{2mm}\end{flushright}\par\medskip}
\newenvironment{remark}{\bigskip\noindent\textbf{Remark: }}
               {\par\bigskip}

\begin{document}

\pagestyle{myheadings}
\markright{The group  Aut and Out of the fundamental group of a closed Sol $3$-manifold}
\title{The group  Aut and Out of the fundamental group of a closed Sol $3$-manifold} 
\date{\today}
\author{Daciberg Lima Gon\c calves\footnote{The author was partially supported by   
FAPESP Projetos Tem\'aticos Topologia Alg\'ebrica, Geom\'etrica e Diferencial 2012/24454-8 and 2016/24707-4 (Brazil)} \and 
   S\'ergio Tadao Martins\footnote{The author was partially supported by
FAPESP: 2013/07510-4.  }}
\maketitle

\begin{abstract}
Let $E$ be the fundamental group of a closed Sol $3$-manifold. We
describe the groups $\Aut(E)$ and $\Out(E)$. We first consider the
case where $E$ is the fundamental group of a torus bundle, and then
the case where $E$ is the fundamental group of a closed Sol
$3$-manifold which is not a torus bundle. The groups are described in
terms of some iterated semi-direct products of well known groups,
where one exception is for $\Out(E)$ for some $E's$, in which case
$\Out(E)$ is described as an extension and a presentation is given.


\end{abstract}

\section{Introduction}

The family of the closed $3$-manifolds which admit the $Sol$-geometry
is one of the $8$ families which appear in the Thurston geometrization
conjecture. All elements of this family are $K(\pi, 1)$ and they are
known to be Haken, i.e., irreducible and containing incompressible
surfaces. So a homotopy equivalence is homotopic to a
homeomorphism. Further, if two homeomorphisms are homotopic, they are
in fact isotopic.  The isotopy classes of base point preserving
homeomorphims of a $Sol$-manifold $N$ are into one-to-one
correspondence with $\Aut(E)$, where $E=\pi_1(N)$.  Furthermore, the
free isotopic classes of homeomorphims of $N$ are into one-to-one
correspondence with $\Out(E)$. See more about the mapping class group
of a 3-manifold and $\Aut(E), \Out(E)$ for $E$ the fundamental group
of the manifold in \cite{HoMc}.Therefore, the computation of $\Aut(E)$
and $\Out(E)$, besides having interest in their own right as an
algebraic problem, is also of interest in the study of $3$-dimensional
manifolds. Also other aspects of this family of manifolds have been
explored. For example, in \cite{sakuma}, the involutions on such
spaces have been studied in detail. In \cite{daci-peter} the Nielsen
fixed point theory has been studied. In \cite{suwawu} the degrees of
self-maps are classified. The present work is related with these three
works.
In this work we compute the groups $\Aut(G)$ and $\Out(G)$ for $G$ the
fundamental group of a closed $Sol$-manifold. It turns out that
$\Out(G)$ is finite. The calculation uses extensively properties of
the linear group $GL_2(\zz)$, some well known, others we explored,
although they should be known for the experts. As a by-product an
application for the fixed point theory of self maps on
$Sol$-manifolds, related with the work \cite{daci-peter}, is obtained,
see remark at the end of subsection~\ref{subsection:eqgl2}.
  
The groups in question can be divided into two classes. One class
consists of the groups of the form $E=(\zz\oplus\zz)\rtimes_{\theta}
\zz$, where $\theta(1)\in GL_2(\zz)$ is an Anosov matrix.  Let $M_0$
be a primitive root of $\theta(1)$. Depending on the matrix $M_0$
these groups are classified into 4 types.  The main results about
these groups are:

  {\bf Theorem 3.5} The group $\Aut(E)$ is given by:\\
\begin{enumerate}[I)]
\item Suppose that $\theta(1)$ is conjugate to a matrix of the form
  $\begin{pmatrix} x & y \\ z & x \end{pmatrix}$. Then $\Aut(E)\cong
  Aut_0(E)\rtimes_w \zz_2,$ where the action of the semi-direct product
  is given by the automorphism $w=w(1_2)\colon
  ((\zz\oplus\zz)\Rtimes\limits_{M_0}\zz)\Rtimes_{\phi} \zz_2 \to
  ((\zz\oplus\zz)\Rtimes\limits_{M_0}\zz)\Rtimes_{\phi} \zz_2$ defined
  by $w(\gamma_-)=\gamma_-$, $w(m, n) =
  -\theta(1) B_0(m, n)$, and
  \[
  w(\gamma_+) = \begin{cases}
    \gamma_+, & \text{if $B_0M_0B_0^{-1} = M_0^{-1}$},\\
    \gamma_-\gamma_+, & \text{if $B_0M_0B_0^{-1} = -M_0^{-1}$}.
  \end{cases}
  \]

\item Suppose that $\theta(1)$ is conjugate to a matrix of the form
  $\begin{pmatrix} x & y \\ y & z
\end{pmatrix}.$
 Then $\Aut(E)\cong
 ((\zz\oplus\zz)\Rtimes\limits_{M_0}\zz)\Rtimes_w \zz_4,$ where
 the action of the semi-direct product is given by the automorphism
 $w=w(1_4):((\zz\oplus\zz)\Rtimes\limits_{M_0}\zz) \to
 ((\zz\oplus\zz)\Rtimes\limits_{M_0}\zz)$ defined by
 $w(m, n) = -\theta(1) B_0(m, n)$ and
  \[
  w(\gamma_+) = \begin{cases}
    \gamma_+, & \text{if $B_0M_0B_0^{-1} = M_0^{-1}$},\\
    \gamma_-\gamma_+, & \text{if $B_0M_0B_0^{-1} = -M_0^{-1}$}.
  \end{cases}
  \]

\item Suppose that $\theta(1)$ is conjugate to a matrix of the form
$A=\begin{pmatrix}
x &  y\\
w-x & w
\end{pmatrix}$. Then $\Aut(E)\cong
  Aut_0(E)\rtimes_w \zz_2,$ where the action of the semi-direct product
  is given by the automorphism $w=w(1_2)\colon
  ((\zz\oplus\zz)\Rtimes\limits_{M_0}\zz)\Rtimes_{\phi} \zz_2 \to
  ((\zz\oplus\zz)\Rtimes\limits_{M_0}\zz)\Rtimes_{\phi} \zz_2$ defined
  by $w(\gamma_-)=\gamma_-$, $w(m, n) =
  -\theta(1) B_0(m, n)$, and
  \[
  w(\gamma_+) = \begin{cases}
    \gamma_+, & \text{if $B_0M_0B_0^{-1} = M_0^{-1}$},\\
    \gamma_-\gamma_+, & \text{if $B_0M_0B_0^{-1} = -M_0^{-1}$}.
  \end{cases}
  \]
  
\end{enumerate}

We also show that  
 $ \Aut_0 (E)\cong
((\zz\oplus\zz)\Rtimes\limits_{M_0}\zz)\Rtimes_{\phi} \zz_2$.

{\bf Theorem 3.6}
  Let $H = \dfrac{\langle \alpha,\beta\rangle}{\langle
    \kappa_d,\kappa_b \rangle} \cong
  \dfrac{\zz\oplus\zz}{(I_2-\theta(1))(\zz\oplus\zz)}$, which is a
  finite group since $\theta(1)$ is Anosov. Considering the four
  cases listed above, we have:
  \begin{enumerate}[I)]
  \item
    \begin{enumerate}[a)]
    \item If $\theta(1) = M_0^\ell$, $\Out(E) \cong (H \rtimes_{M_0}
      \zz_\ell)\rtimes \zz_2$, where the generator of $\zz_2$ acts on
      $H$ by multiplication by $-1$ and on $\zz_\ell$ as the identity.
    \item If $\theta(1) = -M_0^\ell$, $\Out(E) \cong H \rtimes_{M_0}
      \zz_{2\ell}$.
    \end{enumerate}
  \item
    \begin{enumerate}[a)]
    \item If $\theta(1) = M_0^\ell$, $\Out(E) \cong \bigl((H
      \rtimes_{M_0} \zz_\ell)\rtimes \zz_2\bigr)\rtimes \zz_2$, where
      the actions of the two $\zz_2$ factors are induced by the
      conjugations by $\gamma_-$ and $\xi$ in $\Aut(E)$.
    \item If $\theta(1) = -M_0^\ell$, $\Out(E) \cong (H \rtimes_{M_0}
      \zz_{2\ell})\rtimes \zz_2$, where the action of $\zz_2$ is
      induced by the conjugation by $\xi$ in $\Aut(E)$.
    \end{enumerate}
  \item
    \begin{enumerate}[a)]
    \item If $\theta(1) = M_0^\ell$, $\Out(E) = (H\rtimes_{M_0}
      \zz_\ell)\rtimes \zz_4$, where the action of $\zz_2$ is induced
      by the conjugation by $\xi$ in $\Aut(E)$.
    \item If $\theta(1) = -M_0^\ell$, $\Out(E) = (H\rtimes_{M_0}
      \zz_{2\ell})\rtimes \zz_2$, where the action of $\zz_2$ is
      induced by the conjugation by $\xi$ in $\Aut(E)$
       \end{enumerate}
       \item
        \begin{enumerate}[a)]
    \item If $\theta(1) = M_0^\ell$, $\Out(E) \cong \bigl((H
      \rtimes_{M_0} \zz_\ell)\rtimes \zz_2\bigr)\rtimes \zz_2$, where
      the actions of the two $\zz_2$ factors are induced by the
      conjugations by $\gamma_-$ and $\xi$ in $\Aut(E)$.
    \item If $\theta(1) = -M_0^\ell$, $\Out(E) \cong (H \rtimes_{M_0}
      \zz_{2\ell})\rtimes \zz_2$, where the action of $\zz_2$ is
      induced by the conjugation by $\xi$ in $\Aut(E)$.
    \end{enumerate}
  \end{enumerate}

 About the groups of the second class, also known   either as sapphire or semi-torus bundles, we have the following results:
 
 {\bf Theorem 4.10}
  The short exact sequence
  \[
  \begin{tikzcd}
    1 \arrow{r} & \Aut_0(E) \arrow{r} & \Aut(E) \arrow{r} & \zz_2 \arrow{r} & 1
  \end{tikzcd}
  \]
  splits. Hence
  \[
  \Aut(E) \cong \Aut_0(E)\rtimes_\alpha\zz_2 \cong \bigr[\bigl(\zz\oplus\zz)\rtimes_{-I_2}\zz_2\bigr]\rtimes_\omega\zz\bigr]\rtimes_\zeta \zz_2.
  \]

and

{\bf Theorem 4.11}
\begin{enumerate}[I)]
\item If $\Aut_0^1(E)=\emptyset$ we have
\begin{align*}
  \Out(E) = \langle \alpha,\,\beta,\,\rho, \mid \:&
  \alpha^2=\beta^{2s}=\rho^2=1,\\
  &\alpha\beta=\beta\alpha,\, \rho\alpha\rho=\alpha,\, \rho\beta\rho=\beta^{-1}
  \rangle.
\end{align*}
Hence  $\Out(E) \cong (\zz_2\oplus\zz_{2s})\rtimes_{-1}\zz_2$.\\
\item If  $\Aut_0^1(E)\ne\emptyset$
and $\theta(1) = M_0^{\ell_0} = (M_0^\ell)^2$, 
$\det(M_0^\ell) = 1$,  $M_0^\ell = \begin{pmatrix}
  r&s\\t&r\end{pmatrix}$ and the sapphire is defined by
  $B=\begin{pmatrix}r&-t\\-s&r\end{pmatrix}$, we have 
   $\Out(E) \cong
[(\zz_2\oplus\zz_{2t})\rtimes_{-1}\zz_2]\rtimes_\omega \zz_2,$ where 
\begin{align*}
   \omega\alpha\omega^{-1} = \alpha^r\beta^{st},\, \omega\beta\omega^{-1} = \alpha\beta^r,  \omega\rho\omega^{-1} = \alpha^t\beta^{t(r+1)}\rho.
\end{align*}\\
\item If   $\Aut_0^1(E)\ne\emptyset$ and $\theta(1) = -M_0^{\ell_0} = -(M_0^\ell)^2,$
$\det(M_0^\ell) = -1$, $M_0^\ell = \begin{pmatrix}
  r&s\\t&r\end{pmatrix}$ and the sapphire is defined by
  $B=\begin{pmatrix}r&-t\\s&-r\end{pmatrix}$, we have 
 a short exact sequence   

\[
\begin{tikzcd}
  1 \ar[r] & (\zz_2\oplus\zz_{2t})\rtimes_{-1}\zz_2 \ar[r] & \Out(E) \ar[r] & \zz_2 \ar[r] & 1
\end{tikzcd}
\]
and a presentation of $\Out(E)$ is given by
\begin{align*}
  \Out(E) = \langle \alpha,\,\beta,\,\rho,\,\omega \mid \:&
  \alpha^2=\beta^{2t}=\rho^2=1,\\
  & \alpha\beta=\beta\alpha,\, \rho\alpha\rho=\alpha,\, \rho\beta\rho=\beta^{-1},\\
  & \omega^2 = \beta^t\rho,\, \omega\alpha\omega^{-1} = \alpha^r\beta^{st},\, \omega\beta\omega^{-1} = \alpha\beta^r,\\
  & \omega\rho\omega^{-1} = \alpha^t\beta^{t(r+1)}\rho
  \rangle.
\end{align*}
\end{enumerate}

Finally, one remark about the calculations we had to work through:
specially in Section~\ref{section:sapphires}, we used the computer
algebra system Maxima (\texttt{maxima.sourceforge.net}) to
verify many of the identities that we present, since working them by
hand would  be extremely error-prone.

\section{Preliminaries}

The family of closed Sol $3$-manifolds can be divided  into two
types. One type is the subfamily of the torus bundles given by an
Anosov homeomorphism of the torus. Therefore any such manifold is
determined by a matrix $A\in GL_2(\zz)$ and is denoted by
$M_{A}$. Furthermore, from \cite[Lemma 1.1 item 4]{sakuma}
we have that $M_A$ is isomorphic to $M_B$ if and only if
$A$ is either conjugate to $B$ or to $B^{-1}$.

The second type of closed Sol $3$-manifolds are the ones which are
not torus bundle and belong to the subfamily of the so-called sapphires,
following~\cite{morimoto}. They are constructed as follows: for any Anosov matriz $A$ given by
\begin{equation}\label{gluing_map}
\begin{pmatrix}
  r & s \\
  t & u 
\end{pmatrix},
\end{equation}
we obtain a sapphire by gluing two $I$-bundles over the Klein bottle
by a homeomorphism of the boundary, which is a torus, defined by the
matrix $A$. Denote by $K_A$ the resulting
manifold. From~\cite{morimoto} we obtain that $K_A$ is not a torus
bundles if and only if $rstu\ne 0$. Also,~\cite[Theorem~1]{morimoto}
gives the precise description when two such matrices $A$ and  $B$ lead to
two homeomorphic manifolds $K_A$, $K_B$.

Now we recall the main result from \cite{wells}. Let us consider a
short exact sequence of groups
\begin{equation}\label{exactwells0}
\begin{tikzcd}[column sep=1.5em]
\phantom{.}1 \arrow{r} & G \arrow{r}{\iota} & E \arrow{r}{\eta} & \Pi \arrow{r} & 1.
\end{tikzcd}
\end{equation}
Denote    by: {\it   $\Aut(E; G)$ the set of automorphisms of $E$ which, when 
restricted to $G$ define an isomorphism of $G$;   $\mathcal{Z}(G)$ the centre of
$G$;   $\tilde{\alpha}\colon \Pi \to \Out(G)$ the homomorphism given
by the above extension;  $Z^1_{\alpha}(\Pi, \mathcal{Z}G)$ the crossed
homomorphisms.}

\begin{definition}
A pair $(\sigma, \tau) \in \Aut(\Pi) \times \Aut(G)$ is called
\emph{compatible} if $\sigma$ fixes $\Ker(\tilde \alpha)$ and the
automorphism induced by $\sigma$ on $\Pi{\tilde \alpha}$ is the same
as that induced by the inner automorphism of $\Out(G)$ determined by
$\tau$.
\end{definition}
Let $C$ be the set of all compatible pairs. The following Theorem is
proved in~\cite{wells}:
\begin{theorem}\label{theoremwells}
There is a set map $C \to  H^2(\Pi, \mathcal{Z}(G))$ such that the sequence
\begin{equation}\label{exactwells1}
\begin{tikzcd}[column sep=1.5em]
1 \arrow{r} & Z^1_{\alpha}(\Pi,  \mathcal{Z}(G)) \arrow{r} & \Aut(E; G) \arrow{r} &
C \arrow{r} & H^2(\Pi,  \mathcal{Z}(G))
\end{tikzcd}
\end{equation}
is exact. 
\end{theorem}

Sometimes it is more convenient to use the following exact sequence,
which promptly follows from the Theorem above:
\begin{equation}\label{exactwells2}
\begin{tikzcd}[column sep=1.5em]
1 \arrow{r} & Z^1_{\alpha}(\Pi, \mathcal{Z}(G)) \arrow{r} & \Aut(E; G) \arrow{r}{\psi} & \Aut(\Pi)\times \Aut(G),
\end{tikzcd}
\end{equation}
where the last map is the natural obvious one, and it is also a
homomorphism. For more details about the maps and homomorphisms that
appear in the theorem above see~\cite{wells}.
 
\begin{corollary} If the group $G$ has trivial center then we have an isomorphism 
\[
\begin{tikzcd}[column sep=1.5em]
\phantom{.}1 \arrow{r} & \Aut(E; G) \arrow{r}{\cong} & C \arrow{r} &  1.
\end{tikzcd}
\]
\end{corollary}

\begin{corollary}\label{exten}  If the extension (\ref{exactwells0})   is characteristic then  we have the following short exact sequence

\begin{equation}
\begin{tikzcd}[column sep=1.5em]
\phantom{.}1 \arrow{r} & Z^1_{\alpha}(\Pi,  \mathcal{Z}(G)) \arrow{r}{\iota} & Aut(E) \arrow{r}{\psi} &Im(\psi)  \arrow{r} & 1.
\end{tikzcd}
\end{equation}
\end{corollary}

Let us recall how the action of $\im(\psi)$ on $Z^1_{\alpha}(\Pi,
\mathcal{Z}(G))$ is defined: given an element of $\lambda\in
\im(\psi)$, then there is an automorphism $\alpha: E \to E$ which
corresponds to the given element.  Then we must compute $\lambda
*(\beta)=\alpha\circ \beta\circ (\alpha)^{-1}$ where
$\beta=Id+\beta_0$ and $\beta_0$ is a derivation. The result is an
automorphism of $E$ that when restricted to the subgroup $G$ is the
identity and induces the identity on the quotient. So $\lambda
*(\beta)=Id +\beta_0'$ and it is easy to see that
$\beta_0'=\alpha|_{G}^{-1}\circ \beta_0\circ \bar \alpha$.

\section{The groups Aut and Out of the fundamental group of a torus bundle}

Here we consider the fundamental group of  torus bundles where the
homeomorphism used to construct it is Anosov.  It is well known that
these fundamental groups are semi-direct products
$(\zz\oplus\zz)\rtimes_{\theta} \zz$ where the automorphism
$\theta(1)$ is a $2\times 2$ Anosov matrix. To simplify, sometimes by
abuse of notation we denote $(\zz\oplus\zz)\rtimes_{\theta} \zz$ by
$(\zz\oplus\zz)\Rtimes\limits_{\theta(1)}\zz$.

So let $E=(\zz\oplus\zz)\rtimes_{\theta} \zz$, where $\theta(1)\in
GL_2(\zz)$ is an Anosov matrix. The short exact sequence
\[
\begin{tikzcd}[column sep=1.5em]
\phantom{.}1  \arrow{r} &  \zz\oplus\zz \arrow{r} & E \arrow{r} & \zz  \arrow{r} & 1
\end{tikzcd}
\]
splits and the action of $\zz$ on $\zz\oplus\zz$ is given by $\theta(1)$.
 
If we call $d$ and $b$ the generators of $\zz\oplus\zz \trianglelefteq
E$ which correspond to generators of the summands, respectively, and
$v$ an element of $E$ that projects to the generator $1\in \zz$, then
a presentation of $E$ is given by
\begin{equation}\label{torusbundlepresentation}
E = \gen{d,b,v \mid db=bd,\: vdv^{-1} = \theta(1)\begin{pmatrix}1\\0\end{pmatrix},\: vbv^{-1} = \theta(1)\begin{pmatrix}0\\1\end{pmatrix}},
\end{equation}
where we make the identification   $d^mb^n
= \begin{pmatrix}m\\n\end{pmatrix}$ for the elements of
  $\zz\oplus\zz$.

\begin{lemma}\label{extend0}
The subgroup $\zz\oplus\zz$ is characteristic with respect to
automorphisms of $E$. Hence we obtain the exact sequence
\[
\begin{tikzcd}[column sep=1.5em]
\phantom{.}1 \arrow{r} & \zz^1_{\theta}(\zz, \zz\oplus\zz) \arrow{r} & \Aut(E) \arrow{r}{\psi} & \Aut(\zz)\times \Aut(\zz\oplus\zz).
\end{tikzcd}
\]
\end{lemma}

\begin{proof}
The first part follows from \cite{daci-peter}. The second part is a
consequence of Theorem~\ref{theoremwells}. More precisely, it follows
readily from equation~(\ref{exactwells2}).
\end{proof}

In order to compute $\Aut(E)$ we first determine all automorphisms
that induce the identity $\id_\zz$ on the quotient $\zz$. Next we
determine all the automorphisms that induce $-\id_\zz$ on the
quotient, and finally we describe the group structure of $\Aut(E)$ and
of $\Out(E)$.

For $S\in GL_2(\zz)$ let $C(S) = \{M \in GL_2(\zz) \::\:
MS = SM\}$.  It is known (see~\cite[Lemma 1.7]{sakuma})
that there is a matrix $M_0\in GL_2(\zz)$ such that
\begin{equation}\label{eq:M_0}
C(S) = \{\pm M_0^\ell \::\: \ell \in \zz\} \cong \zz_2\oplus\zz,
\end{equation}
if $S$ is neither the identity matrix $I_2$ nor $-I_2$.  If
$S$ is of infinite order then the number of possible matrices
$M_0$ are exactly $4$. Namely, given $M_0$, they are $\{ \pm
M_0^{\pm1}\}$.

\begin{proposition}\label{ident}
The intersection of image of $\psi$ with the subgroup $\{\id_\zz\}
\times \Aut(\zz\oplus\zz)$ is the subgroup generated by $(\id_\zz, M_0)$
and $(\id_\zz,-\id_{\zz\oplus\zz})$. If we denote by $\Aut_0(E)$ the
elements of $\Aut(E)$ which project to elements of the form $(\id_{\zz},
\tau)$, we have the short exact sequence
\[
\begin{tikzcd}[column sep=1.5em]
\phantom{.}1  \arrow{r} & (\zz\oplus\zz)\Rtimes\limits_{M_0}\zz \arrow{r} &  \Aut_0 (E) \arrow{r} & \zz_2 \arrow{r} & 1,
\end{tikzcd}
\]
where the map $\Aut_0 (E) \to \zz_2 $ is determined by the map sending
$(\id_\zz, M_0)$ to $0_2$ and $(\id_\zz,-\id_{\zz\oplus\zz})$ to
$1_2$. Further, the sequence splits, so $ \Aut_0 (E)\cong
((\zz\oplus\zz)\Rtimes\limits_{M_0}\zz)\Rtimes_{\phi} \zz_2$, and with
respect to the splitting given by $s(1_2)= ( \id_{\zz},
-\id_{\zz\oplus\zz}),$ the action is given by $\phi(m,n,t)=(-m,-n,t)$.
\end{proposition}

\begin{proof}
The equation $M\theta(1)=\theta(1) M$ is precisely the condition for the
existence of an automorphism on the semi-direct product which induces
the automorphisms $\id_\zz \in \Aut(\zz)$ and $M\in
\Aut(\zz\oplus\zz)$. The short exact sequence given by
Theorem~\ref{theoremwells} and the fact that $Z^1_{\theta}(\zz,
\zz\oplus\zz) = \zz\oplus\zz$, together with equation~(\ref{eq:M_0}),
give us the short exact sequence
\[
\begin{tikzcd}[column sep=3em]
\phantom{.}1  \arrow{r} & \zz\oplus\zz \arrow{r} &  \Aut_0 (E) \arrow{r}{\left.\psi\right|_{\Aut_0(E)}} & \zz_2\oplus\zz \arrow{r} & 1.
\end{tikzcd}
\]
The automorphisms $\alpha$, $\beta\in \Aut_0(E)$ defined by
\begin{align*}
&\alpha(d) = d, & \beta(d) = d,\\
&\alpha(b) = b, & \beta(b) = b,\\
&\alpha(v) = dv, & \beta(v) = bv
\end{align*}
generate $\ker(\left.\psi\right|_{\Aut_0(E)}) \cong \zz\oplus\zz$. The
automorphism $\gamma_+\in\Aut_0(E)$ given by
\[
\gamma_+(d^mb^n) = M_0\begin{pmatrix}m\\n\end{pmatrix}, \quad \gamma_+(v) = v
\]
is projected by $\psi$ onto $(\id_\zz,M_0)$ and the automorphism
$\gamma_-\in \Aut_0(E)$ defined by
\[
\gamma_-(d^mb^n) = d^{-m}b^{-n}, \quad \gamma_-(v) = v
\]
is projected onto $(\id_\zz,-\id_{\zz\oplus\zz})$. The subgroup of
$\Aut_0(E)$ generated by $\alpha$, $\beta$ and $\gamma_+$ is
isomorphic to $(\zz\oplus\zz)\rtimes_{M_0} \zz$. This can be readily
seen since the subgroup generated by $\gamma_+$ is infinite cyclic and
$(\gamma_+)(\alpha^x\beta^y)(\gamma_+)^{-1} =
M_0(\alpha^x\beta^y)$. Finally, the subgroup generated by $\alpha$,
$\beta$ and $\gamma_+$ has index $2$ in $\Aut_0(E)$, and this gives
the short exact sequence of the statement. If one wants a presentation
for $\Aut_0(E)$, a straightforward computation shows that
$(\gamma_-)\alpha(\gamma_-)^{-1} = \alpha^{-1}$,
$(\gamma_-)\beta(\gamma_-)^{-1} = \beta^{-1}$,
$(\gamma_-)(\gamma_+)(\gamma_-)^{-1} = \gamma_+$ and, of course,
$(\gamma_-)^2 = 1$.
The ``further'' part is straighforward.
\end{proof}

  
Although the explict calculation of a generator $M_0$ may not be
straightforward,  the  elements  $(\id_\zz, \theta(m))$, for $m\in Z\backslash\{0\}$
are non-trivial elements which belong to the image of $\psi$ besides $(\id_\zz,
-\id_{\zz\oplus\zz})$. This is not the case for the elements where the
first coordinate is $-\id_\zz$, since  the set of such elements may or may not 
be empty.  In order to find elements in the image of $\psi$ in
this case, this is equivalent to ask if for the given $\theta(1) \in
GL_2(\zz)$ we can find $B\in GL_2(\zz)$ such that $B\theta(1)
B^{-1}=\theta^{-1}(1)$. Following \cite{sakuma} we denote by
$R(\theta(1))$ the set of $B\in GL_2(\zz)$ such that $B\theta(1)
B^{-1}=\theta(1)^{-1}$. See the Appendix for a detailed analyzis of
the set $R(\theta(1))$, where we give necessary and sufficient
condition for the set $R(\theta(1))$ to be  not empty.


\begin{lemma}\label{minusident} There is   a short exact sequence
\[
\begin{tikzcd}[column sep=1.5em]
\phantom{.}1 \arrow{r} & \Aut_0(E) \arrow{r} & \Aut(E) \arrow{r}{\psi} &  \zz_2 \arrow{r} & 1,
\end{tikzcd}
\] if $R(\theta(1))$ is not empty, otherwise we have  $\Aut_0(E)  \cong
  \Aut(E)$. 
\end{lemma}
\begin{proof} The proof is straightforward.
\end{proof}

Now we calculate a presentation of $\Aut(E)$ assuming that
$R(\theta(1))$ is not empty, and let $B_0\in R(\theta(1))$. We use
~\cite[Proposition~1, page~139]{Jon}, and consider an automorphism
$\xi$ such that restricted to $\zz\oplus\zz$ is given by a matrix
$B_0\in R(\theta(1))$, and such that  induces $-\id_\zz$ on the
quotient $\zz$.
The matrix $B_0^2$ commutes with $\theta(1)$. Therefore if we choose
another element $\xi_1$ with similar properties as $\xi$, it must be
of the form $\xi_1=\varepsilon M_0^k \xi$ where $\varepsilon
\in\{I_2,-I_2\}$. It follows that
\[
\xi_1^2=\varepsilon M_0^k \xi \varepsilon M_0^k \xi=(\varepsilon)^2 M_0^k \xi M_0^k \xi^{-1}\xi^2=\xi^2,
\]
so $\xi^2$ is independent of the choice of $B_0\in R(\theta(1))$.  In
fact we can show:

\begin{lemma}
The set $R(\theta(1))$ is non empty if and only if the matrix
$\theta(1)$ has determinant $1$ and is conjugate to one of the   matrices  of
the form $A=\begin{pmatrix} x & y \\ z & x
\end{pmatrix}, $
or to a matrix of the form
$A=\begin{pmatrix}
x & y \\
y & z
\end{pmatrix},$
or to a matrix of the form
$A=\begin{pmatrix}
x &  y\\
w-x & w
\end{pmatrix}$.
In the first and third cases we have $B_0^2=I_2$ and in the second
case we have $B_0^2=-I_2$.
\end{lemma}
\begin{proof}
This follows promptly from Proposition \ref{eqanti} since $B_0$ is
conjugate to $\begin{pmatrix} 1 & 0 \\ 0 & -1
\end{pmatrix}, $ 
 $\begin{pmatrix} 1 & 1 \\ 0 & -1
\end{pmatrix},$ in the first and third cases, respectively,  and conjugate to 
 $\begin{pmatrix}
0 & -1 \\
1 & 0
\end{pmatrix}, $
in the second  case. 
\end{proof}
So to complete a presentation of the group it suffices to describe a
choice of a matrix $B_0$ and the action of the automorphism $\xi\in
\Aut(E)$ on $\Aut_0(E)$ by conjugation. The choice of $B_0$ can be
made using the appendix and then we can get a presentation of
$\Aut(E)$.

\begin{theorem}\label{mainres} The group $\Aut(E)$ is given by:\\
\begin{enumerate}[I)]
\item Suppose that $\theta(1)$ is conjugate to a matrix of the form
  $\begin{pmatrix} x & y \\ z & x \end{pmatrix}$. Then $\Aut(E)\cong
  Aut_0(E)\rtimes_w \zz_2,$ where the action of the semi-direct product
  is given by the automorphism $w=w(1_2)\colon
  ((\zz\oplus\zz)\Rtimes\limits_{M_0}\zz)\Rtimes_{\phi} \zz_2 \to
  ((\zz\oplus\zz)\Rtimes\limits_{M_0}\zz)\Rtimes_{\phi} \zz_2$ defined
  by $w(\gamma_-)=\gamma_-$, $w(m, n) =
  -\theta(1) B_0(m, n)$, and
  \[
  w(\gamma_+) = \begin{cases}
    \gamma_+, & \text{if $B_0M_0B_0^{-1} = M_0^{-1}$},\\
    \gamma_-\gamma_+, & \text{if $B_0M_0B_0^{-1} = -M_0^{-1}$}.
  \end{cases}
  \]

\item Suppose that $\theta(1)$ is conjugate to a matrix of the form
  $\begin{pmatrix} x & y \\ y & z
\end{pmatrix}.$
 Then $\Aut(E)\cong
 ((\zz\oplus\zz)\Rtimes\limits_{M_0}\zz)\Rtimes_w \zz_4,$ where
 the action of the semidirect product is given by the automorphism
 $w=w(1_4):((\zz\oplus\zz)\Rtimes\limits_{M_0}\zz) \to
 ((\zz\oplus\zz)\Rtimes\limits_{M_0}\zz)$ defined by
 $w(m, n) = -\theta(1) B_0(m, n)$ and
  \[
  w(\gamma_+) = \begin{cases}
    \gamma_+, & \text{if $B_0M_0B_0^{-1} = M_0^{-1}$},\\
    \gamma_-\gamma_+, & \text{if $B_0M_0B_0^{-1} = -M_0^{-1}$}.
  \end{cases}
  \]

\item Suppose that $\theta(1)$ is conjugate to a matrix of the form
$A=\begin{pmatrix}
x &  y\\
w-x & w
\end{pmatrix}$. Then $\Aut(E)\cong
  Aut_0(E)\rtimes_w \zz_2,$ where the action of the semi-direct product
  is given by the automorphism $w=w(1_2)\colon
  ((\zz\oplus\zz)\Rtimes\limits_{M_0}\zz)\Rtimes_{\phi} \zz_2 \to
  ((\zz\oplus\zz)\Rtimes\limits_{M_0}\zz)\Rtimes_{\phi} \zz_2$ defined
  by $w(\gamma_-)=\gamma_-$, $w(m, n) =
  -\theta(1) B_0(m, n)$, and
  \[
  w(\gamma_+) = \begin{cases}
    \gamma_+, & \text{if $B_0M_0B_0^{-1} = M_0^{-1}$},\\
    \gamma_-\gamma_+, & \text{if $B_0M_0B_0^{-1} = -M_0^{-1}$}.
  \end{cases}
  \]
  
\end{enumerate}
\end{theorem}

\begin{proof} 
  \begin{enumerate}[I)]
  \item Choose $B_0 = \begin{pmatrix}1 & 0\\0
    & -1\end{pmatrix}$ and $\xi(d^mb^n) = d^mb^{-n}$, $\xi(v) =
    v^{-1}$. Then
    \[
    \xi(\alpha^m\beta^n)\xi^{-1}(v)=d^rb^sv,
    \]
     where $(r,s)=  -\theta(1) B_0(m, n)$, 
    and $\xi\gamma_+\xi^{-1}(d^mb^n) = \pm M_0^{-1}(d^mb^n)$, hence
    $\xi\gamma_+\xi^{-1} \in \{\gamma_+^{-1},
    \gamma_-\gamma_+^{-1}\}$. We have $\xi^2=\id_E$ and
    $\xi\gamma_-\xi^{-1} = \gamma_-$, hence $\Aut(E) =
    \Aut_0(E)\rtimes \zz_2$ with the action given in the statement.
  \item Choose $B_0 = \begin{pmatrix}0 & -1\\1
    & 0\end{pmatrix}$ and $\xi(d^mb^n) = d^{-n}b^m$, $\xi(v) =
    v^{-1}$. Then
    \[
    \xi(\alpha^m\beta^n)\xi^{-1}(v)=d^rb^sv,
    \]
    where $(r,s)=  -\theta(1) B_0(m, n)$, and we have 
    $\xi\alpha\xi^{-1}=\theta(b^{-1})v$ and
    $\xi\beta\xi^{-1} = \theta(d)v$. We also have
    $\xi\gamma_+\xi^{-1}(d^mb^n) = \pm M_0^{-1}(d^mb^n)$ and
    $\xi\gamma_+\xi^{-1}(v) = v$, hence $\xi\gamma_+\xi^{-1} \in
    \{\gamma_+^{-1}, \gamma_-\gamma_+^{-1}\}$. Finally,
    $\xi^2=\gamma_-$ and we have $\Aut(E) \cong
    ((\zz\oplus\zz)\rtimes_{M_0}\zz)\rtimes\zz_4$, where
    $\zz_4$ is generated by the class of $\xi$.
  \item Choose $B_0 = \begin{pmatrix}1 & 1\\0
     & -1\end{pmatrix}$ and $\xi(d^{m}b^{n}) = d^{m+n}b^{-n}$, $\xi(v) =
    v^{-1}$. Then
    \[
    \xi(\alpha^m\beta^n)\xi^{-1}(v)=d^rb^sv,
    \]
    where $(r,s)= -\theta(1) B_0(m, n)$, and
    $\xi\gamma_+\xi^{-1}(d^mb^n) = \pm M_0^{-1}(d^mb^n)$, hence
    $\xi\gamma_+\xi^{-1} \in \{\gamma_+^{-1},
    \gamma_-\gamma_+^{-1}\}$. We have $\xi^2=\id_E$ and
    $\xi\gamma_-\xi^{-1} = \gamma_-$, hence $\Aut(E) =
    \Aut_0(E)\rtimes \zz_2$ with the action given in the statement.
  \end{enumerate}
\end{proof}

Now we calculate $\Out(E)$. For $g\in E$, we denote by $\kappa_g\in
\Inn(E)$ the conjugation by $g$, that is, $\kappa_g(x) = gxg^{-1}$. Let
$\theta(1) = \begin{pmatrix}r & s \\t & u\end{pmatrix}$; an easy
  calculation shows that $\kappa_d = \alpha^{1-r}\beta^{-t}$,
  $\kappa_b = \alpha^{-s}\beta^{1-u}$, so the subgroup $\langle
  \kappa_d,\kappa_b\rangle \le \langle \alpha,\beta \rangle \cong
  \zz\oplus\zz $ is isomorphic to $(I_2-\theta(1))(\zz\oplus\zz)$. We
  also have
\[
\kappa_v = 
\begin{cases}
  \gamma_+^\ell, & \text{if $\theta(1) = M_0^\ell$},\\
  \gamma_-\gamma_+^\ell, & \text{if $\theta(1) = -M_0^\ell$},
\end{cases}
\]

\bigskip

Observe that  every inner automorphism of
$E=(\zz\oplus\zz)\rtimes_{\theta}\zz$ induces the identity on the
quotient $\zz$. Now we compute $\Out_0(E)=\Aut_0(E)/\Inn(E)$ and
$\Out(E)=\Aut(E)/\Inn(E)$.  We consider 4 cases.
\begin{enumerate}[I)]
\item $\Aut_0(E) = \Aut(E)$;
\item $\Aut_0(E)\ne \Aut(E)$ and $\theta(1)$ is conjugate to $\begin{pmatrix}x & y \\ z & x\end{pmatrix}$.
\item $\Aut_0(E)\ne \Aut(E)$ and $\theta(1)$ is conjugate to $\begin{pmatrix}x & y \\ y & z\end{pmatrix}$.
\item $\Aut_0(E)\ne \Aut(E)$ and $\theta(1)$ is conjugate to $\begin{pmatrix}x & y \\ w-x & w\end{pmatrix}$.
\end{enumerate}

\begin{theorem}\label{mainres1}
  Let $H = \dfrac{\langle \alpha,\beta\rangle}{\langle
    \kappa_d,\kappa_b \rangle} \cong
  \dfrac{\zz\oplus\zz}{(I_2-\theta(1))(\zz\oplus\zz)}$, which is a
  finite group since $\theta(1)$ is Anosov. Considering the four
  cases listed above, we have:
  \begin{enumerate}[I)]
  \item
    \begin{enumerate}[a)]
    \item If $\theta(1) = M_0^\ell$, $\Out(E) \cong (H \rtimes_{M_0}
      \zz_\ell)\rtimes \zz_2$, where the generator of $\zz_2$ acts on
      $H$ by multiplication by $-1$ and on $\zz_\ell$ as the identity.
    \item If $\theta(1) = -M_0^\ell$, $\Out(E) \cong H \rtimes_{M_0}
      \zz_{2\ell}$.
    \end{enumerate}
  \item
    \begin{enumerate}[a)]
    \item If $\theta(1) = M_0^\ell$, $\Out(E) \cong \bigl((H
      \rtimes_{M_0} \zz_\ell)\rtimes \zz_2\bigr)\rtimes \zz_2$, where
      the actions of the two $\zz_2$ factors are induced by the
      conjugations by $\gamma_-$ and $\xi$ in $\Aut(E)$.
    \item If $\theta(1) = -M_0^\ell$, $\Out(E) \cong (H \rtimes_{M_0}
      \zz_{2\ell})\rtimes \zz_2$, where the action of $\zz_2$ is
      induced by the conjugation by $\xi$ in $\Aut(E)$.
    \end{enumerate}
  \item
    \begin{enumerate}[a)]
    \item If $\theta(1) = M_0^\ell$, $\Out(E) = (H\rtimes_{M_0}
      \zz_\ell)\rtimes \zz_4$, where the action of $\zz_2$ is induced
      by the conjugation by $\xi$ in $\Aut(E)$.
    \item If $\theta(1) = -M_0^\ell$, $\Out(E) = (H\rtimes_{M_0}
      \zz_{2\ell})\rtimes \zz_2$, where the action of $\zz_2$ is
      induced by the conjugation by $\xi$ in $\Aut(E).$
       \end{enumerate}
       \item
        \begin{enumerate}[a)]
    \item If $\theta(1) = M_0^\ell$, $\Out(E) \cong \bigl((H
      \rtimes_{M_0} \zz_\ell)\rtimes \zz_2\bigr)\rtimes \zz_2$, where
      the actions of the two $\zz_2$ factors are induced by the
      conjugations by $\gamma_-$ and $\xi$ in $\Aut(E)$.
    \item If $\theta(1) = -M_0^\ell$, $\Out(E) \cong (H \rtimes_{M_0}
      \zz_{2\ell})\rtimes \zz_2$, where the action of $\zz_2$ is
      induced by the conjugation by $\xi$ in $\Aut(E)$.
    \end{enumerate}
  \end{enumerate}
\end{theorem}

\begin{proof}
  The analyses of the four  cases are similar.
  \begin{enumerate}[I)]
  \item Let $\theta(1) = \begin{pmatrix}r & s \\t & u\end{pmatrix}$.
    \begin{enumerate}[a)]
    \item A presentation for $\Out(E)$ is given by
      \begin{align*}
        \Out(E) = \langle \alpha,\,\beta,\,\gamma_+,\,\gamma_- \mid \:& \alpha\beta=\beta\alpha,\, \gamma_+(\alpha^m\beta^n)\gamma_+^{-1} = M_0(\alpha^m\beta^n), \\
        & \gamma_-^2 = 1,\, \gamma_-\alpha\gamma_-^{-1} = \alpha^{-1},\, \gamma_-\beta\gamma_-^{-1} = \beta^{-1},\, \gamma_-\gamma_+\gamma_-^{-1} = \gamma_+^{-1}, \\
        &\alpha^{1-r}\beta^{-t} = \alpha^{-s}\beta^{1-u} = 1,\, \gamma_+^\ell = 1\rangle,
      \end{align*}
      from which the statement follows.
    \item In this case, a presentatation of $\Out(E)$ is given by
      \begin{align*}
        \Out(E) &= \langle \alpha,\,\beta,\,\gamma_+,\,\gamma_- \mid \: \alpha\beta=\beta\alpha,\, \gamma_+(\alpha^m\beta^n)\gamma_+^{-1} = M_0(\alpha^m\beta^n), \\
        &\phantom{= \langle \alpha,\,\beta,\,\gamma_+,\,\gamma_- \mid \: } \gamma_-^2 = 1,\, \gamma_-\alpha\gamma_-^{-1} = \alpha^{-1},\, \gamma_-\beta\gamma_-^{-1} = \beta^{-1},\, \gamma_-\gamma_+\gamma_-^{-1} = \gamma_+^{-1}, \\
        &\phantom{= \langle \alpha,\,\beta,\,\gamma_+,\,\gamma_- \mid \: }\alpha^{1-r}\beta^{-t} = \alpha^{-s}\beta^{1-u} = 1,\, \gamma_- = \gamma_+^\ell\rangle \\
        \phantom{\Out(E)} &= \langle \alpha,\, \beta,\, \gamma_+ \mid \: \alpha\beta=\beta\alpha,\, \gamma_+^{2\ell} = 1,\, \gamma_+(\alpha^m\beta^n)\gamma_+^{-1} = M_0(\alpha^m\beta^n),\\
        &\phantom{= \langle \alpha,\, \beta,\, \gamma_+ \mid \: }\alpha^{1-r}\beta^{-t} = \alpha^{-s}\beta^{1-u} = 1\rangle,
      \end{align*}
      from which the statement follows.
    \end{enumerate}
  \item This case is analogous to the previous one.
  \item
    \begin{enumerate}[a)]
    \item This item is also analogous to case I.
    \item Let $\theta(1) = \begin{pmatrix}x & y \\ y & z\end{pmatrix}$
      and $B_0 = \begin{pmatrix}0 & -1\\1 & 0\end{pmatrix}$. In this
        case, a presentation of $\Out(E)$ is
      \begin{align*}
        \Out(E) &= \langle \alpha,\, \beta,\, \gamma_+,\, \xi \mid\: \alpha\beta=\beta\alpha,\, \alpha^{1-x}\beta^{-y} = \alpha^{-y}\beta^{1-z} = 1, \\
        &\phantom{= \langle \alpha,\, \beta,\, \gamma_+,\, \xi \mid\: }\gamma_+(\alpha^m\beta^n)\gamma_+^{-1} = M_0(\alpha^m\beta^n),\\
        &\phantom{= \langle \alpha,\, \beta,\, \gamma_+,\, \xi \mid\: }\xi(\alpha^m\beta^n)\xi^{-1} = -\theta B_0(\alpha^m\beta^n),\, \xi\gamma_+\xi^{-1} = \xi^2\gamma_+^{-1},\\
        &\phantom{= \langle \alpha,\, \beta,\, \gamma_+,\, \xi \mid\: }\xi^4=1,\, \xi^2\gamma_+^\ell = 1\rangle \\
        &=\langle \alpha,\, \beta,\, \gamma_+,\, \xi \mid\: \alpha\beta=\beta\alpha,\, \alpha^{1-x}\beta^{-y} = \alpha^{-y}\beta^{1-z} = 1, \\
        &\phantom{= \langle \alpha,\, \beta,\, \gamma_+,\, \xi \mid\: }\gamma_+(\alpha^m\beta^n)\gamma_+^{-1} = M_0(\alpha^m\beta^n),\\
        &\phantom{= \langle \alpha,\, \beta,\, \gamma_+,\, \xi \mid\: }\xi(\alpha^m\beta^n)\xi^{-1} = -\theta B_0(\alpha^m\beta^n),\, \gamma_+\xi^{-1} = \xi\gamma_+^{-1},\\
        &\phantom{= \langle \alpha,\, \beta,\, \gamma_+,\, \xi \mid\: }\xi^2=\gamma_+^\ell,\, \gamma_+^{2\ell} = 1\rangle,
      \end{align*}
      from which the statement follows, since the exact sequence
      \[
      \begin{tikzcd}
        1 \ar[r] & H\rtimes_{M_0} \zz_{2\ell} \ar[r] & \Out(E) \ar[r] & \zz_2 = \langle \overline\xi\rangle \ar[r] & 1
      \end{tikzcd}
      \]
      splits, with a section given by $\overline\xi \mapsto \xi\gamma_+^{-1}$.
    \end{enumerate}
     \item This case is analogous to the case II, which in turn is analogous to I.
  \end{enumerate}
    
\end{proof}

\section{The groups Aut and Out for the Sapphire Sol-manifold}\label{section:sapphires}

Here we follow the classification of the sapphire manifolds given by
Morimoto in~\cite{morimoto}. For each matrix
\begin{equation}\label{gluing_map}
          B=\begin{pmatrix}
                 r & s \\
                 t & u
           \end{pmatrix}\in GL_2(\zz),
\end{equation} 
a $3$-manifold is constructed using a homeomorphism $h$ of the torus
$T$, where $h$ induces on the fundamental group of $T$ the
homomorphism given by the matrix $B$. We know from~\cite[Theorem
  1]{morimoto} when two matrices $B_1, B_2$ provide, up to
homeomorphism, the same manifold, and it follows that we can assume
without loss of generality that $\det(B) = ru-st = 1$. So, from now
on, our matrix $B$ is fixed and of determinant $1$. Also, if the
sapphire has the Sol geometry and is not a torus bundle,
from~\cite{morimoto} we must have $rstu\ne 0$.

Let $E$ be the fundamental group of the sapphire defined by the matrix
$B$. The group $\Aut(E)$ is closely related to $\Aut(E')$, where $E'$
is the fundamental group of a certain torus bundle. More precisely,
by~\cite[Lemma 3.3]{daci-peter} we have the following short exact
sequence where the kernel is characteristic with respect to
automorphism of $E$:
\[
1 \to (\zz\oplus\zz)\Rtimes_{\theta}\zz \to E \to \zz_2 \to 1.
\]
The matrix $\theta=\theta(1)$ is given by
\begin{equation}\label{gluing_map}
          \begin{pmatrix}
                 ru+st & -2rt \\
                 -2su & ru+st
           \end{pmatrix}.
\end{equation}

One presentation of $E$ is
\begin{equation}\label{sapphirepresentation}
E = \!\!\begin{array}[t]{l}
 \langle \,d,b,v,a \mid {}db=bd,\: vdv^{-1} = \theta(d),\: vbv^{-1} = \theta(b), \\
\phantom{\langle \,d,b,v,a \mid {}}a^2=d,\: ab=b^{-1}a,\: ava^{-1} = d^{r-ru-st}b^{s-2su}v^{-1}\,\rangle,
\end{array}
\end{equation}
where $d$ and $b$ are the generators of the subgroup $(\zz\oplus\zz)$
of $(\zz\oplus\zz)\rtimes_\theta\zz$, $v\in
(\zz\oplus\zz)\rtimes_\theta\zz$ is such that its class generates the
quotient $\zz = ((\zz\oplus\zz)\rtimes_\theta\zz)/(\zz\oplus\zz)$, and
$a$ is a remaining generator of $E$ that projects onto the generator
of $\zz_2$.

\begin{lemma} The map $\psi: \Aut(E) \to \Aut((\zz\oplus\zz)\rtimes_{\theta}\zz)$ 
is injective.
\end{lemma}
\begin{proof}
This follows promptly from the exact sequence given by Theorem \ref {theoremwells} and the fact that the center of
$(\zz\oplus\zz)\rtimes_{\theta}\zz$ is trivial.
\end{proof}

Our task is to decide which automorphisms of
$(\zz\oplus\zz)\rtimes_{\theta}\zz$ extend to the group $E$, which
amounts to deciding which automorphisms of
$(\zz\oplus\zz)\rtimes_{\theta}\zz$ also preserve the last three
relations of~(\ref{sapphirepresentation}). We will consider two types
of automorphisms of $(\zz\oplus\zz)\rtimes_{\theta}\zz$:  Type I
are the automorphisms which induce the identity map on the quotient
$[(\zz\oplus\zz)\rtimes_{\theta}\zz]/(\zz\oplus\zz)=\zz$, and  Type
II are the automorphisms which induce minus the identity map on the
same quotient.

\begin{remark}
The torus bundles $(\zz\oplus\zz)\Rtimes_{\theta}\zz$ where
$\theta(1)$ is given by (\ref{gluing_map}), by Proposition
\ref{eqanti},  always admit automorphisms of Type II since the matrix
of the gluing map of the torus bundles has the two elements of the
diagonal equal.
\end{remark}

We first address the case of Type I maps. In this case it follows from
the classification of the automorphisms, of
$(\zz\oplus\zz)\rtimes_{\theta}\zz$, that a such automorphism
when restricted to $\zz\oplus\zz$ is an automorphism of the form
$\delta M_0^\ell$ for some $\ell\in \zz$, $\delta\in \{\pm1\}$ and
some primitive Anosov matrix $M_0$ given by~\cite[Lemma
  1.7]{sakuma}. Recall that the matrix $B$ used to construct the
Sapphire has $\det(B)=1$. So let $\varphi\in
\Aut((\zz\oplus\zz)\rtimes_\theta\zz)$ be a Type I automorphism that
we attempt to extend to an automorphism of $E$ (that we also call
$\varphi$ by an abuse of notation) by setting
\begin{align}\label{eq:phiattemptI}
\varphi(d^xb^y) &= \delta M_0^\ell(d^xb^y), \quad  \delta \in \{\pm 1\},  \notag\\
\varphi(v) &= d^pb^qv, \\
\varphi(a) &= d^mb^nv^ka \notag,
\end{align}
for certain integers $\ell$, $p$, $q$, $m$, $n$, $k$. The next lemma
answers when such a map preserves the relation $ab=b^{-1}a$.

\begin{lemma}[Fundamental I]\label{lemma:typeIcondition}
Let $\varphi\in \Aut((\zz\oplus\zz)\rtimes_\theta\zz)$ be a Type I
automorphism. If we attempt to extend $\varphi$ to an automorphism of
$E$ as in equation~(\ref{eq:phiattemptI}), then the extension
$\varphi$ satisfies $\varphi(ab) = \varphi(b^{-1}a)$ if, and only if,
$2\ell = k\ell_0$ and $\varepsilon^k\det(M_0)^\ell=1$, where
$\ell_0\in\zz$ is the integer such that $\theta(1) = \varepsilon
M_0^{\ell_0}$, $\varepsilon \in \{\pm 1\}$.
\end{lemma}
\begin{proof}
Let $P = \begin{pmatrix}x & z\\ y & w\end{pmatrix}\in SL_2(\rr)$ be a
  matrix such that $P^{-1}M_0 P$ is diagonal. We claim that $xw+yz=0$
  (which readily implies $xw=1/2$ and $yz=-1/2$). One way to see this
  is the following: if $P^{-1}M_0P$ is diagonal, then so is
  $P^{-1}\theta P$. Now a simple computation of the eigenvectors of
  $\theta$ (observing that the diagonal entries of $\theta$ are equal)
  yields the claim.

We have
\begin{gather*}
\varphi(ab) = \varphi(b^{-1}a) \Iff\\
d^mb^nv^ka\begin{pmatrix}\delta & 0\\ 0 & \delta\end{pmatrix}M_0^\ell(b) = \begin{pmatrix}\delta & 0\\ 0 & \delta\end{pmatrix}M_0^\ell(b^{-1})d^mb^nv^ka \Iff\\
v^ka\begin{pmatrix}\delta & 0\\ 0 & \delta\end{pmatrix}M_0^\ell(b) = \begin{pmatrix}\delta & 0\\ 0 & \delta\end{pmatrix}M_0^\ell(b^{-1})v^ka \Iff\\
\theta^k\begin{pmatrix}\delta & 0 \\ 0 & -\delta\end{pmatrix}M_0^\ell\begin{pmatrix}0 \\ 1\end{pmatrix} = \begin{pmatrix}\delta & 0\\ 0 & \delta\end{pmatrix}M_0^\ell\begin{pmatrix}1 & 0 \\ 0 & -1\end{pmatrix}\begin{pmatrix}0 \\ 1\end{pmatrix} \Iff\\
\left[\begin{pmatrix}1 & 0 \\ 0 & -1\end{pmatrix}M_0^{-\ell}\begin{pmatrix}\delta & 0\\ 0 & \delta\end{pmatrix}\theta^k\begin{pmatrix}\delta & 0 \\ 0 & -\delta\end{pmatrix}M_0^\ell\right]\begin{pmatrix}0 \\ 1\end{pmatrix} = \begin{pmatrix}0 \\ 1\end{pmatrix}.
\end{gather*}
Let $N = \begin{pmatrix}1 & 0 \\ 0 & -1\end{pmatrix}M_0^{-\ell}\begin{pmatrix}\delta & 0\\ 0 & \delta\end{pmatrix}\theta^k\begin{pmatrix}\delta & 0 \\ 0 & -\delta\end{pmatrix}M_0^\ell$.
We have $\det(N) = \det(\theta)^k = 1^k = 1$ and, since $1$ is an
eigenvalue of $N$, the characteristic polynomial of $N$ is
$(\lambda-1)^2$. We also have that $\begin{pmatrix}0
  \\ 1\end{pmatrix}$ is an eigenvector corresponding to the eigenvalue
  $1$, so we actually have
\[
N = \begin{pmatrix}
1 & 0 \\
e & 1
\end{pmatrix}
\]
for some $e\in\zz$. Hence we can write
\[
\begin{pmatrix}1 & 0 \\ 0 & -1\end{pmatrix}M_0^{-\ell}\begin{pmatrix}\delta & 0\\ 0 & \delta\end{pmatrix}\theta^k\begin{pmatrix}\delta & 0 \\ 0 & -\delta\end{pmatrix}M_0^\ell = \begin{pmatrix}1 & 0 \\e & 1\end{pmatrix}.
\]
Now let $P = \begin{pmatrix}x & z\\ y & w\end{pmatrix}\in SL_2(\rr)$
  be such that $P^{-1}M_0P = \begin{pmatrix}\lambda_1 & 0 \\ 0 &
    \lambda_2\end{pmatrix}$. Multiplying the equation above by
    $P^{-1}$ on the left and by $P$ on the right, we get
\[
\varepsilon^k\begin{pmatrix}
\lambda_1^\ell\lambda_2^{k\ell_0-\ell} & 0 \\
0 & \lambda_2\lambda_1^{k\ell_0-\ell} \\
\end{pmatrix} = 
\begin{pmatrix}
(1-exz) & -ez^2\\
ex^2 & (1+exz)
\end{pmatrix},
\]
which implies $e=0$ and $\lambda_1^\ell\lambda_2^{k\ell_0-\ell} =
\lambda_2^\ell\lambda_1^{k\ell_0-\ell}=\pm 1$. But $M_0$ is an Anosov matrix,
so $|\lambda_1| \ne 1 \ne |\lambda_2|$ and $|\lambda_1\lambda_2| =
1$. Since we have
\[
|\lambda_1^\ell\lambda_2^{k\ell_0-\ell}|=1 \Rightarrow |\lambda_2^{k\ell_0-2\ell}|=1,
\]
we get $k\ell_0-2\ell = 0$. Now, from the above matrix equality, we
must have $\varepsilon^k(\lambda_1\lambda_2)^\ell = 1 \iff
\varepsilon^k\det(M_0)^\ell =1$.
\end{proof}

Next we look at the other two relations that $\varphi$ in
equation~(\ref{eq:phiattemptI}) must satisfy to be an element of
$\Aut(E)$. We first analyze the case $k$ even, starting when $k=0$.
Still using the notation of equation~(\ref{eq:phiattemptI}), let's
denote by $\Aut_0(E)$ the automorphisms of $E$ that induce a type I
automorphism of $(\zz\oplus \zz)\rtimes_\theta \zz$, and let
\begin{equation}\label{eq:aut0k}
\Aut_0^k(E) = \{\varphi\in \Aut_0(E) \:\mid\: \varphi(a) = d^mb^nv^ka \text{ for some $m$, $n\in\zz$}\}.
\end{equation}
It is easy to see that $\varphi_{k_1}\in \Aut_0^{k_1}(E)$,
$\varphi_{k_2}\in \Aut_0^{k_2}(E) \then \varphi_{k_1}\varphi_{k_2}\in
\Aut_0^{k_1+k_2}(E)$ and $\varphi_k\in \Aut_0^k(E)\then \varphi_k^{-1}
\in \Aut_0^{-k}(E)$, hence $\Aut_0^0(E)$ is a normal subgroup of
$\Aut_0(E)$. It is this subgroup that we describe now. By
Lemma~\ref{lemma:typeIcondition}, we know that $k=0$ implies $\ell=0$.

\begin{proposition}\label{proposition:G0}
An element $\varphi\in \Aut_0^0(E)$ has one of the following two
forms:
\begin{enumerate}
\item
  \begin{align*}
    \varphi(d^xb^y) &= d^xb^y, \\
    \varphi(a) &= b^na, \\
    \varphi(v) &= d^{rtn+ct}b^{-stn-cu}v,
  \end{align*}
  where $c$, $n$ are any integers.
\item
  \begin{align*}
    \varphi(d^xb^y) &= d^{-x}b^{-y}, \\
    \varphi(a) &= d^{-1}b^na, \\
    \varphi(v) &= d^{r(tn-1+u)+ct}b^{-s(tn-1+u)-cu}v,
  \end{align*}
  where $c$, $n$ are any integers.
\end{enumerate}
\end{proposition}
\begin{proof}
  The relation $a^2=d$ yields $d^{2m+1} = d^{\pm 1}$, hence $m=0$ or
  $m=-1$ according to the sign. In the first case, corresponding to
  $m=0$, the relation $ava^{-1} = d^{r-ru-st}b^{s-2su}v^{-1}$ yields
  the linear system (in the unknowns $p$, $q$)
  \[
  \left(\theta^{-1}+ \begin{pmatrix}1&0\\0&-1\end{pmatrix}\right)\begin{pmatrix}p\\q\end{pmatrix} =
      (\theta^{-1}-I_2)\begin{pmatrix}0\\n\end{pmatrix},
  \]
  which is equivalent to $up+tq = tn$ and has $p = rtn+ct$, $q =
  -stn-cu$ as the general integer solution.

  In the second case, corresponding to $m=-1$, the linear system obtained is
  \[
  \left(\theta^{-1}+ \begin{pmatrix}1&0\\0&-1\end{pmatrix}\right)\begin{pmatrix}p\\q\end{pmatrix} =
      -2\begin{pmatrix}r-ru-st\\s-2su\end{pmatrix}+(\theta^{-1}-I_2)\begin{pmatrix}-1\\n\end{pmatrix},
  \]
  which is equivalent to $up+tq = nt-1+u$ and has $p = r(tn-1+u)+ct$, $q =
  -s(tn-1+u)-cu$ as the general integer solution.
\end{proof}

In the previous proposition, an element $\varphi \in \Aut_0^0(E)$ of
the first type is entirely determined by the integers $c$ and $n$. It
is immediate to check that the map
\begin{align*}
  \Xi \colon \zz\oplus\zz & \to \Aut_0^0(E) \\
  (n,c) &\mapsto \varphi\colon E \to E \\
  &\phantom{\mapsto \varphi\colon} \varphi(d^xb^y) = d^xb^y \\
  &\phantom{\mapsto \varphi\colon} \varphi(a) = b^na \\
  &\phantom{\mapsto \varphi\colon} \varphi(v) = d^{rtn+ct}b^{-stn-cu}v
\end{align*}
is an injective group homomorphism.

\begin{theorem}\label{theorem:aut00}
  There is a split exact sequence
  \[
  \begin{tikzcd}
    0 \ar[r] & \zz\oplus\zz \ar[r,"\Xi"] & \Aut_0^0(E) \ar[r] & \zz_2 \ar[r] & 0
  \end{tikzcd}
  \]
  and $\Aut_0^0(E) \cong (\zz\oplus\zz)\rtimes_{-I_2}\zz_2$.
\end{theorem}

\begin{proof}
  The homomorphism $\Aut_0^0(E) \to \zz_2$ is defined by mapping to
  the generator of $\zz_2$ exactly those automorphisms $\varphi$ such
  that $\varphi(d^xb^y) = d^{-x}b^{-y}$, so the existence of the short
  exact sequence is clear. Also, a simple calculation shows that any
  $\rho\in \Aut_0^0(E)$ that satisfies $\rho(d^xb^y) = d^{-x}b^{-y}$
  also satisfies $\rho^2=1_E$, so the sequence splits. In particular,
  for
  \begin{align*}
    \rho(d^xb^y) &= d^{-x}b^{-y} & \varphi(d^xb^y) &= d^xb^y \\
    \rho(a) &= d^{-1}a & \varphi(a) &= b^na \\
    \rho(v) &= d^{r(u-1)}b^{s(1-u)}v & \varphi(v) &= d^{rtn+ct}b^{-stn-cu}v, \\
  \end{align*}
  we have
  \begin{align*}
    \rho\varphi\rho(d^xb^y) &= d^xb^y, \\
    \rho\varphi\rho(a) &= b^{-n}a, \\
    \rho\varphi\rho(v) &= d^{-rtn-ct}b^{stn+cu}v,
  \end{align*}
  so the map $\zz_2 \to \Aut(\zz\oplus\zz)$ takes the generator of
  $\zz_2$ to the matrix $-I_2$.
\end{proof}

Now we investigate the cosets $\Aut_0^k(E)$ of $\Aut_0^0(E)$ in
$\Aut_0(E)$. The next result shows that $\Aut_0^k(E)\ne\emptyset$
when $k$ is even.

\begin{theorem}
If $k=2$, an extension $\varphi$ as in equation~(\ref{eq:phiattemptI})
exists.
\end{theorem}

\begin{proof}
If the extension $\varphi$ exists, we have $\ell=\ell_0$, $M_0^\ell =\varepsilon\theta$ ($\varepsilon=\pm 1$) and the
relation $a^2=d$ gives us
\begin{gather*}
d^mb^nv^2ad^mb^nv^2a = \varepsilon\theta(d) \iff \\
d^mb^nv^2d^mb^{-n}(ava^{-1})^2a^2 = \varepsilon\theta(d) \\
d^mb^n\theta^2(d^mb^{-n})v^2(d^{r-ru-st}b^{s-2su}v^{-1})^2d = \varepsilon\theta(d) \\
\left(I_2+\theta^2\begin{pmatrix}1 & 0\\0&-1\end{pmatrix}\right)\begin{pmatrix}m\\n\end{pmatrix} +
(\theta+\theta^2)\begin{pmatrix}r-ru-st\\s-2su\end{pmatrix} = (\varepsilon \theta-I_2)\begin{pmatrix}1\\0\end{pmatrix}.
\end{gather*}
We put the above system (with unknowns $m$, $n$) in the matrix form
$A\begin{pmatrix}m\\n\end{pmatrix} = \begin{pmatrix}b_1 \\ b_2\end{pmatrix}$ and analyze two cases,
corresponding to $\varepsilon = \pm1$. In both cases, we have
\[
A = \begin{pmatrix}
  (ru+st)^2+4rstu+1 & 4rt(ru+st) \\
  -4su(ru+st) & -(ru+st)^2-4rstu+1
\end{pmatrix},
\]
and note that $\det(A)=0$. Hence the system will have (rational)
solutions if, and only if, $A_{11}b_2-A_{21}b_1=0$. Now we write
$d=ru-st=1$ and notice that, in both cases, corresponding to the
$+$ or $-$ signs, we have
\[
A_{11}b_2-A_{21}b_1 = 2(d-1)s(2st-1) = 0.
\]
Assuming $\varepsilon=1$, the linear system is equivalent to
\[
(2st+1)m+(2rt)n = 1-r,
\]
which has integer solutions since $\gcd(2st+1,2rt)=1$. The general
integer solution $(m,n)$ is given by
\begin{align*}
m &= (ru+st-r)+c(2rt)\\
n &= (s-2su)-c(2st+1).
\end{align*}

Still assuming $\varepsilon=1$, we now consider the relation $ava^{-1}
= d^{r-ru-st}b^{s-2su}v^{-1}$, which yields
\[
d^mb^nv^2ad^pb^qva^{-1}v^{-2}d^{-m}b^{-n} = \theta(d^{r-ru-st}b^{s-2su})v^{-1}d^{-p}b^{-q},
\]
from which we obtain the following linear system in the unknowns $p$ and $q$:
\[
\left(I_2-\theta^{-1}\right)\begin{pmatrix}m\\n\end{pmatrix} +
\left(\theta^2\begin{pmatrix}1&0\\0&-1\end{pmatrix} + \theta^{-1}\right)
\begin{pmatrix}p\\q\end{pmatrix} =
(\theta-\theta^2)\begin{pmatrix}r-ru-st\\s-2su\end{pmatrix}.
\]
Writing the above system in the matrix form
$C\begin{pmatrix}p\\q\end{pmatrix}=\begin{pmatrix}e_1 \\ e_2 \end{pmatrix}$, we have $\det(C)=0$ and hence
the system has rational solutions if, and only if,
$C_{11}e_2-C_{21}e_1 = 0$. We get
\begin{align*}
  C_{11}e_2-C_{21}e_1 &= (8s^2t^2+10st+2)(2msu-2su+4s^2t+2nst+2s) - \\
  & ((2-4r)st+2mst+2nrt)(-8s^2t-2s)u.
\end{align*}
Substituting the expressions for $m$ and $n$ in the above expression,
it factors as $4s(4st + 1)(u + 2cst^2 - 3st + ct - 1)(ru - st - 1) =
0$, so the linear system has rational solutions and is equivalent to
\begin{gather*}
(8s^2t^2+10st+2)p+(8rst^2+6rt)q = t(-4rs-2cr) \\
u(4ru-3)p+t(4ru-1)q = -t(2s+c),
\end{gather*}
that has integer solutions for any $c\in\zz$ since
$\gcd(u(4ru-3),t(4ru-1))=1$. In fact
\begin{gather*}
  \left|
  \begin{array}{l}
    d\mid 4ru^2-3u\\
    d\mid 4rut-t
  \end{array}
  \right.\then
  \left|
  \begin{array}{l}
    d\mid 4ru^2t-3ut\\
    d\mid 4rut-t
  \end{array}
  \right.\then
  \left|
  \begin{array}{l}
    d\mid 2ut\\
    d\mid 4rut-t
  \end{array}
  \right.\then
  \left|
  \begin{array}{l}
    d\mid 4rut\\
    d\mid 4rut-t
  \end{array}
  \right.\then \\
  \left|
  \begin{array}{l}
    d\mid t\\
    d\mid u(4ru-3)
  \end{array}
  \right.\stackrel{\gcd(u,t)=1}{\then}
  \left|
  \begin{array}{l}
    d\mid t\\
    d\mid 4ru-3
  \end{array}
  \right.\then
  \left|
  \begin{array}{l}
    d\mid t\\
    d\mid 4(1+st)-3 = 4st+1
  \end{array}
  \right.\then
  d\mid 1.
\end{gather*}

Finally, in the case $\varepsilon=-1 \iff M_0^\ell = -\theta$, we also
have extensions $\varphi$ since they are all obtained by composing
with an appropriate element of $\Aut_0^0(E)$.
\end{proof}

The investigation of the cosets $\Aut_0^k(E)$ for $k$ odd is greatly
simplified now that the even case is understood. For $k$ odd, we have
$\Aut_0^k(E)\ne\emptyset \iff \Aut_0^1(E)\ne\emptyset$ We start with a
very simple lemma.

\begin{lemma}
If neither $\theta(1)$ nor $-\theta(1)$ has a square root, there is no
automorphism $\varphi$ with $k$ odd.
\end{lemma}
\begin{proof}
By the fundamental lemma, if there is an automorphism with $k$ odd,
then we have $\ell_0$ even. Since $\theta(1)=\pm M_0^{\ell_0}$, there  follows
that $M_0^{\ell}$ is a square root of either $\theta(1)$ nor
$-\theta(1)$. So the result follows.
\end{proof}

Let $k=1$, which by the fundamental lemma implies $\ell_0=2\ell$. As a
consequence of the previous lemma, we need to consider two cases, the
first being when $\theta(1)$ admits square root, and the second when
$-\theta(1)$ admits square root.

Let us set up some notation. Because the two elements of the diagonal
of $\theta(1)$ are the same, this is also the case with $M_0$. So let
us write
$M_0^{\ell} = \begin{pmatrix} r & s
  \\ t & r
  \end{pmatrix}$,
  where $M_0$ is Anosov, so

$M_0^{\ell_0} = \begin{pmatrix} r^2+st & 2rs
  \\ 2rt & r^2+st
  \end{pmatrix}.$
  
According to Lemma~\ref{lemma:typeIcondition}, the two cases to
analyze are as follows, knowing that the relation $ab=b^{-1}a$ is
satisfied:

\begin{enumerate}[1)]
\item $\theta(1)=M_0^{\ell_0}$ (which implies $\det(M_0^\ell)=1$),
  $\varphi(d^xb^y)= \delta M_0^\ell(d^xb^y)$, $\delta = \pm
  1$.
\item $\theta(1)=-M_0^{\ell_0}$ (which implies $\det(M_0^{\ell})=-1$),
  $\varphi(d^xb^y)= \delta M_0^\ell(d^xb^y)$, $\delta = \pm
  1$.
\end{enumerate}

We will show below that in the two cases above we have extensions
$\varphi\in \Aut_0^1(E)$ for $\delta = 1$. Composing with
elements of $\Aut_0^0(E)$, this means we also have extensions with
$\delta=-1$. So we will assume that $\delta=1$ in our
analysis.

Consider the second relation $a^2=d$. To study the equation
$\varphi(a)^2=\varphi(d)$, we need to have our sapphire defined, since
we use the relation $ava^{-1}$, which was not the case for the first
relation. Let $\theta(1)$ be as given in case $1$ and consider the
sapphire given by the matrix
\[
B = \begin{pmatrix} r & -t
  \\  -s & r
\end{pmatrix}.
\]
The two fold cover is the torus bundle given by $\theta(1)$ and
observe that $\det(N)=1$.

\begin{theorem}\label{help3I}
  In case $1$, there is a $\varphi\in \Aut_0^1(E)$.
\end{theorem}
\begin{proof}
  Let $\varphi$ be as in~(\ref{eq:phiattemptI}) with $k=\delta=1$. The
  relation $\varphi^2(a)=\varphi(d)$ is equivalent to the homogeneous
  linear system
  \[
  \begin{pmatrix}
    r^2+st+1 & -2rs \\
    2rt & -r^2-st+1
  \end{pmatrix}
  \begin{pmatrix}
    m\\n
  \end{pmatrix} =
  \begin{pmatrix}
    0\\0
  \end{pmatrix},
  \]
  so we take the obvious solution $m=n=0$. The last relation to be verified is then
  \[
  \varphi(ava^{-1}) = \varphi(d^{r-r^2-st}b^{-t+2rt}v^{-1}),
  \]
  which is equivalent to $2(r^2+st)p-(4rs)q = 2-2r$ and admits $p =
  r^2-r+st$, $q = t(2r-1)$ as a solution. One explicit element
  $\omega\in \Aut_0^1(E)$ is then given by
  \begin{align*}
    \omega(d^xb^y) &= M(d^xb^y), \\
    \omega(v) &= d^{r^2-r+st}b^{t(2r-1)}v, \\
    \omega(a) &= va,
  \end{align*}
  where $M = M_0^\ell = \begin{pmatrix}r & s\\ t&r\end{pmatrix}.$
\end{proof}

Case 2 is similar: we consider the sapphire given by the matrix
\[
B = \begin{pmatrix}
  r & -t \\
  s & -r
\end{pmatrix}
\]
and then we can state:

\begin{theorem}\label{help3II}
  In case $2$, there is also a $\varphi\in \Aut_0^1(E)$.
\end{theorem}

The proof of case $2$ is completely analogous and will be ommitted. We
will, however, write down an explicit element $\omega\in \Aut_0^1(E)$
obtained in this case:
\begin{align*}
  \omega(d^xb^y) &= M(d^xb^y), \\
  \omega(v) &= d^{-r}b^{-t}v, \\
  \omega(a) &= va,
\end{align*}
where $M = M_0^\ell = \begin{pmatrix}r & s\\ t&r\end{pmatrix}.$

%
%
%
%
%
%
%

\bigskip

The coset $\Aut_0^1(E)$ of $\Aut_0^0(E)$ may be empty or not. In any
case, we have a homomorphism $\Aut_0(E) \to \zz$ given by
\[
\varphi \in \Aut_0^k(E) \mapsto k,
\]
whose image is either $\zz$ or $2\zz \cong \zz$, so there is a split
short exact sequence
\[
\begin{tikzcd}
  1 \ar[r] & \Aut_0^0(E) \ar[r] & \Aut_0(E) \ar[r] & \zz \ar[r] & 1
\end{tikzcd}
\]
and we have

\begin{theorem}\label{theorem:aut0}
  $\Aut_0(E) \cong \Aut_0^0(E) \rtimes_\omega \zz \cong
  [(\zz\oplus\zz)\rtimes_{-I_2}\zz_2]\rtimes_\omega\zz$.
\end{theorem}

\begin{proof}
  Following theorem~\ref{theorem:aut00} and considering the morphism
  $\Xi$ just before it, the standard generators of
  $(\zz\oplus\zz)\rtimes_{-I_2}\zz_2$ corresponding to $(1,0,0)$,
  $(0,1,0)$ and $(0,0,1)$ are $\alpha,\beta,\rho \in \Aut_0^0(E)$
  given by
  \begin{align*}
    \alpha(d^xb^y) &= d^xb^y & \beta(d^xb^y) &= d^xb^y & \rho(d^xb^y) &= d^{-x}b^{-y} \\
    \alpha(v) &= d^{rt}b^{-st}v & \beta(v) &= d^tb^{-u}v & \rho(v) &= d^{r(u-1)}b^{s(1-u)}v \\
    \alpha(a) &= ba & \beta(a) &= a & \rho(a) &= d^{-1}a.
  \end{align*}
  We have two cases to consider. If $\Aut_0^1(E) = \emptyset$, then
  $\omega$ is determined by an element of $\Aut_0^2(E)$. One such
  element, that we also call $\omega$, is given by
  \begin{align*}
    \omega(d^xb^y) &= \theta(d^xb^y)\\
    \omega(v) &= d^{4rst}b^{2s-4rsu}v \\
    \omega(a) &= d^{ru+st-r}b^{s-2su}v^2a.
  \end{align*}
  Its inverse is given by
  \begin{align*}
    \omega^{-1}(d^xb^y) &= \theta^{-1}(d^xb^y) \\
    \omega^{-1}(v) &= b^{2s}v \\
    \omega^{-1}(a) &= d^{-8rstu+r-1}b^{4su-8rsu^2+s}v^{-2}a.
  \end{align*}
  The calculation of $\omega\alpha\omega^{-1}$,
  $\omega\beta\omega^{-1}$ and $\omega\rho\omega^{-1}$ is extensive
  but straightforward. We get
  \begin{align*}
    (\omega\alpha\omega^{-1})(d^xb^y) &= d^xb^y \\
    (\omega\alpha\omega^{-1})(v) &= d^{rt(4st+1)}b^{-st(4st+3)}v & \then && \omega\alpha\omega^{-1} &= \alpha^{ru+st}\beta^{2rst} \\
    (\omega\alpha\omega^{-1})(a) &= b^{ru+st}(a) \\
    \\
    (\omega\beta\omega^{-1})(d^xb^y) &= d^xb^y \\
    (\omega\beta\omega^{-1})(v) &= d^{t(4st+3)}b^{-u(4st+1)}v & \then && \omega\beta\omega^{-1} &= \alpha^{2u}\beta^{ru+st}\\
    (\omega\beta\omega^{-1})(a) &= b^{2u}a \\
    \\
    (\omega\rho\omega^{-1})(d^xb^y) &= d^{-x}b^{-y} \\
    (\omega\rho\omega^{-1})(v) &= d^{4s^2t^2+4rst+5st-r+1}b^{-s(1+4st+3u+4stu)}v & \then && \omega\rho\omega^{-1} &= \alpha^{2s(u+1)}\beta^{2rs(u+1)}\rho\\
    (\omega\rho\omega^{-1})(a) &= d^{-1}b^{2s+2su}a
  \end{align*}

  The second case is $\Aut_0^1(E)\ne \emptyset$, which breaks into  two
  subcases. The first is $\theta(1) = M_0^{\ell_0} = (M_0^\ell)^2$,
  $\det(M_0^\ell) = 1$, where $M_0^\ell = \begin{pmatrix}
    r&s\\t&r\end{pmatrix}$ and the sapphire is defined by
    $B=\begin{pmatrix}r&-t\\-s&r\end{pmatrix}$. The maps $\alpha$,
    $\beta$, $\rho$ and $\omega$ analogous to the ones considered in
    the first case are given by

  \begin{align*}
    \alpha(d^xb^y) &= d^xb^y & \beta(d^xb^y) &= d^xb^y & \rho(d^xb^y) &= d^{-x}b^{-y} \\
    \alpha(v) &= d^{-rs}b^{-st}v & \beta(v) &= d^{-s}b^{-r}v & \rho(v) &= d^{r(r-1)}b^{t(r-1)}v \\
    \alpha(a) &= ba & \beta(a) &=a & \rho(a) &= d^{-1}a
  \end{align*}
  \begin{align*}
    \omega(d^xb^y) &= M_0^\ell(d^xb^y) & \omega^{-1}(d^xb^y) &= M_0^{-\ell}(d^xb^y) \\
    \omega(v) &= d^{r^2-r+st}b^{t(2r-1)}v & \omega^{-1}(v) &= d^{1-r}b^{-t}v \\
    \omega(a) &= va & \omega^{-1}(a) &= d^{-2st+r-1}b^{t(2r-1)}v^{-1}a.
  \end{align*}

  What we get in this case is
  \begin{align*}
    \omega\alpha\omega^{-1} &= \alpha^r\beta^{st} \\
    \omega\beta\omega^{-1} &= \alpha\beta^r \\
    \omega\rho\omega^{-1} &= \alpha^{-t}\beta^{-t(r+1)}\rho.
  \end{align*}

  Finally, the second subcase is $\theta(1) = -M_0^{\ell_0} =
  -(M_0^\ell)^2$, $\det(M_0^\ell) = -1$, where $M_0^\ell
  = \begin{pmatrix} r&s\\t&r\end{pmatrix}$ and the sapphire is defined
    by $B=\begin{pmatrix}r&-t\\s&-r\end{pmatrix}$. The maps $\alpha$,
    $\beta$, $\rho$ and $\omega$ are now given by

  \begin{align*}
    \alpha(d^xb^y) &= d^xb^y & \beta(d^xb^y) &= d^xb^y & \rho(d^xb^y) &= d^{-x}b^{-y} \\
    \alpha(v) &= d^{rs}b^{st}v & \beta(v) &= d^{s}b^{r}v & \rho(v) &= d^{-r(r+1)}b^{-t(r+1)}v \\
    \alpha(a) &= ba & \beta(a) &=a & \rho(a) &= d^{-1}a
  \end{align*}
  \begin{align*}
    \omega(d^xb^y) &= M_0^\ell(d^xb^y) & \omega^{-1}(d^xb^y) &= M_0^{-\ell}(d^xb^y) \\
    \omega(v) &= d^{-r}b^{-t}v & \omega^{-1}(v) &= dv \\
    \omega(a) &= va & \omega^{-1}(a) &= d^{2st-1}b^{-2rt}v^{-1}a.
  \end{align*}

  Once again we get
  \begin{align*}
    \omega\alpha\omega^{-1} &= \alpha^r\beta^{st} \\
    \omega\beta\omega^{-1} &= \alpha\beta^r \\
    \omega\rho\omega^{-1} &= \alpha^{-t}\beta^{-t(r+1)}\rho.
  \end{align*}
\end{proof}

Let's now consider the type II automorphisms. First we observe that
the group $(\zz\oplus\zz)\rtimes_\theta\zz$, where
\[
\theta(1) = \begin{pmatrix}
  ru+st & -2rt \\
  -2su & ru+st
\end{pmatrix},
\]
always admits automorphisms that induces $-id_\zz$ on the quotient
$[(\zz\oplus\zz)\rtimes_\theta\zz]/(\zz\oplus\zz)$, and the
restrictions to $(\zz\oplus\zz)$ can always be written on the form
$\pm B_0M_0^\ell$ for some $\ell\in\zz$, where $B_0$ satisfies
$B_0\theta(1) B_0^{-1} = \theta^{-1}$. A solution for $B_0$ on the
latter equation is given by
\[
B_0 = \begin{pmatrix}
1 & 0 \\
0 & -1
\end{pmatrix},
\]
see Proposition~\ref{eqanti} in the Appendix.

If $\varphi_1$ and $\varphi_2$ are two automorphisms of $E$ whose
restrictions to $(\zz\oplus\zz)\rtimes_\theta\zz$ are Type II
automorphisms, then the restriction of $\varphi_1\varphi_2^{-1}$ to
$(\zz\oplus\zz)\rtimes_\theta\zz$ is a Type I automorphism. Hence, if
we denote by $\Aut_0(E)$ the elements of $\Aut(E)$ that induce Type I
automorphisms on $(\zz\oplus\zz)\rtimes_\theta\zz$, we have the short
exact sequence of groups
\begin{equation}\label{eq:autE}
\begin{tikzcd}
1 \arrow{r} & \Aut_0(E) \arrow{r} & \Aut(E) \arrow{r} & \zz_2 \arrow{r} & 1.
\end{tikzcd}
\end{equation}
The generator of the quotient $\zz_2$ is the class of any element of
$\Aut(E)$ that induces a Type II automorphism on
$(\zz\oplus\zz)\rtimes_\theta\zz$. One such family of automorphisms is
given by
\begin{align*}
&\zeta \colon E \to E\\
&\zeta(d^xb^y) = \begin{pmatrix}1 & 0 \\ 0 & -1\end{pmatrix}\begin{pmatrix}x \\ y\end{pmatrix}, \\
&\zeta(v) = d^{p}b^{q}v^{-1}, \\
&\zeta(a) = a,
\end{align*}
where we must find all $p$ and $q$ for which the map above extends to an
automorphism. The calculation leads to
\begin{align*}
  p & =r-(ru+st)+\lambda \frac{(ru+st-1)}{\gcd\{-2su, ru+st-1\}} = r-(ru+st)+\lambda t, \\
  q & =(s-2su)-\lambda \frac{(-2su)}{\gcd\{-2su, ru+st-1\}} = (s-2su)+\lambda u,
\end{align*}
for any $\lambda\in \zz$, since $ru+st-1 = 2st$ and $\gcd(2su,2st)=2s$.

For $\lambda=s$, we have $p=r-ru$ and $q=s-su$ and a direct
computation shows that $\zeta^2=1_E$. Hence we have

\begin{theorem}\label{theorem:autE}
  The short exact sequence
  \[
  \begin{tikzcd}
    1 \arrow{r} & \Aut_0(E) \arrow{r} & \Aut(E) \arrow{r} & \zz_2 \arrow{r} & 1
  \end{tikzcd}
  \]
  splits, hence
  \[
  \Aut(E) \cong \Aut_0(E)\rtimes_\alpha\zz_2 \cong \bigr[\bigl(\zz\oplus\zz)\rtimes_{-I_2}\zz_2\bigr]\rtimes_\omega\zz\bigr]\rtimes_\zeta \zz_2.
  \]
\end{theorem}

\begin{proof}
  The structure of $\Aut(E)$ is clear. We just need to compute the action of $\zeta$, where
  \begin{align*}
  \zeta(d^xb^y) &= d^xb^{-y} \\
  \zeta(v) &= d^{r-ru}b^{s-su}v^{-1} \\
  \zeta(a) &= a
  \end{align*}
  is an automorphism of type II of order $2$. We will follow the
  notation from theorem~\ref{theorem:aut0}. When $\Aut_0^1(E) =
  \emptyset$, we get
  \[
  \zeta\alpha\zeta = \alpha^{-1}, \quad \zeta\beta\zeta = \beta, \quad
  \zeta\rho\zeta = \rho, \quad \zeta\omega\zeta = \alpha^{-2s(u+1)}\beta^{2rs(u+1)}\omega^{-1}.
  \]
  When $\Aut_0^1(E)\ne\emptyset$, we have the same two subcases
  considered in theorem~\ref{theorem:aut0}. In the first subcase, we
  have $\theta(1) = M_0^{\ell_0} = (M_0^\ell)^2$, $\det(M_0^\ell) =
  1$, where $M_0^\ell = \begin{pmatrix} r&s\\t&r\end{pmatrix}$ and the
    sapphire is defined by $B=\begin{pmatrix}r&-t\\-s&r\end{pmatrix}$.
  The map $\zeta$ is given by
  \begin{align*}
  \zeta(d^xb^y) &= d^xb^{-y}, \\
  \zeta(v) &= d^{r(1-r)}b^{t(r-1)}v^{-1}, \\
  \zeta(a) &= a,    
  \end{align*}
  and we get
  \[
  \zeta\alpha\zeta = \alpha^{-1}, \quad \zeta\beta\zeta = \beta, \quad
  \zeta\rho\zeta = \rho, \quad \zeta\omega\zeta = \alpha^{t}\beta^{-t(r+1)}\omega^{-1}.
  \]
  In the second subcase, we have $\theta(1) = -M_0^{\ell_0} =
  -(M_0^\ell)^2$, $\det(M_0^\ell) = -1$, where $M_0^\ell
  = \begin{pmatrix} r&s\\t&r\end{pmatrix}$ and the sapphire is defined
    by $B=\begin{pmatrix}r&-t\\s&-r\end{pmatrix}$. The map $\zeta$ is given by
  \begin{align*}
  \zeta(d^xb^y) &= d^xb^{-y}, \\
  \zeta(v) &= d^{r(r+1)}b^{-t(r+1)}v^{-1}, \\
  \zeta(a) &= a,    
  \end{align*}
  and we get
  \[
  \zeta\alpha\zeta = \alpha^{-1}, \quad \zeta\beta\zeta = \beta, \quad
  \zeta\rho\zeta = \rho, \quad \zeta\omega\zeta = \alpha^{t}\beta^{-rt}\rho\omega^{-1}.
  \]
\end{proof}

Now we compute $\Out(E)$. If $g\in E$, let's denote by $\kappa_g \in
\Inn(E)$ the conjugation by $g$: $\kappa_g(x) = gxg^{-1}$. We will use
the notations from theorems~\ref{theorem:aut0} and~\ref{theorem:autE}.

If $\Aut_0^1(E)=\emptyset$, we have
\[
\kappa_d = \beta^{-2s}, \quad \kappa_b = \alpha^2, \quad
\kappa_v = \alpha^{-2s}\beta^{-2rs}\omega, \quad \kappa_a = \beta^{-s}\zeta,
\]
hence a presentation of $\Out(E)$ is given by
\begin{align*}
  \Out(E) = \langle \alpha,\,\beta,\,\rho,\,\omega,\,\zeta \mid \:&
  \alpha\beta=\beta\alpha,\, \rho^2=1,\, \rho\alpha\rho=\alpha^{-1},\, \rho\beta\rho=\beta^{-1},\\
  &\omega\alpha\omega^{-1} = \alpha^{ru+st}\beta^{2rst},\, \omega\beta\omega^{-1} = \alpha^{2s(u+1)}\beta^{2rs(u+1)}\rho,\\
  &\zeta^2=1,\, \zeta\alpha\zeta=\alpha^{-1},\, \zeta\beta\zeta=\beta,\\
  &\zeta\rho\zeta=\rho,\, \zeta\omega\zeta=\alpha^{-2s(u+1)}\beta^{2rs(u+1)}\omega^{-1},\\
  &\beta^{-2s}=\alpha^2=\alpha^{-2s}\beta^{-2rs}\omega=\beta^{-s}\zeta=1
  \rangle.
\end{align*}

We have $\zeta=\beta^s$, and $\omega=\alpha^{2s}\beta^{2rs} = 1$, so
\begin{align*}
  \Out(E) = \langle \alpha,\,\beta,\,\rho, \mid \:&
  \alpha^2=\beta^{2s}=\rho^2=1,\\
  &\alpha\beta=\beta\alpha,\, \rho\alpha\rho=\alpha,\, \rho\beta\rho=\beta^{-1}
  \rangle
\end{align*}
and $\Out(E) \cong (\zz_2\oplus\zz_{2s})\rtimes_{-1}\zz_2$.

\bigskip

If $\Aut_0^1(E)\ne\emptyset$, as usual we have two subcases: in the
first one, when $\theta(1) = M_0^{\ell_0} = (M_0^\ell)^2$,
$\det(M_0^\ell) = 1$, $M_0^\ell = \begin{pmatrix}
  r&s\\t&r\end{pmatrix}$ and the sapphire is defined by
  $B=\begin{pmatrix}r&-t\\-s&r\end{pmatrix}$, we get
\[
\kappa_d = \beta^{2t}, \quad \kappa_b = \alpha^2, \quad
\kappa_v = \alpha^{2t}\beta^{2rt}\omega^2, \quad \kappa_a = \beta^{t}\zeta,
\]
hence a presentation of $\Out(E) \cong
[(\zz_2\oplus\zz_{2t})\rtimes_{-1}\zz_2]\rtimes_\omega \zz_2$ is given by
\begin{align*}
  \Out(E) = \langle \alpha,\,\beta,\,\rho,\,\omega \mid \:&
  \alpha^2=\beta^{2t}=\rho^2=\omega^2=1,\\
  & \alpha\beta=\beta\alpha,\, \rho\alpha\rho=\alpha,\, \rho\beta\rho=\beta^{-1},\\
  & \omega\alpha\omega^{-1} = \alpha^r\beta^{st},\, \omega\beta\omega^{-1} = \alpha\beta^r,\\
  & \omega\rho\omega^{-1} = \alpha^t\beta^{t(r+1)}\rho
  \rangle.
\end{align*}

\bigskip
Finally, when $\theta(1) = -M_0^{\ell_0} = -(M_0^\ell)^2$,
$\det(M_0^\ell) = -1$, $M_0^\ell = \begin{pmatrix}
  r&s\\t&r\end{pmatrix}$ and the sapphire is defined by
  $B=\begin{pmatrix}r&-t\\s&-r\end{pmatrix}$, we have
\[
\kappa_d = \beta^{2t}, \quad \kappa_b = \alpha^2, \quad
\kappa_v = \beta^{-t}\rho\omega^2, \quad \kappa_a = \beta^{t}\zeta,
\]
hence $\Out(E)$ fits in the short exact sequence
\[
\begin{tikzcd}
  1 \ar[r] & (\zz_2\oplus\zz_{2t})\rtimes_{-1}\zz_2 \ar[r] & \Out(E) \ar[r] & \zz_2 \ar[r] & 1
\end{tikzcd}
\]
and a presentation is given by
\begin{align*}
  \Out(E) = \langle \alpha,\,\beta,\,\rho,\,\omega \mid \:&
  \alpha^2=\beta^{2t}=\rho^2=1,\\
  & \alpha\beta=\beta\alpha,\, \rho\alpha\rho=\alpha,\, \rho\beta\rho=\beta^{-1},\\
  & \omega^2 = \beta^t\rho,\, \omega\alpha\omega^{-1} = \alpha^r\beta^{st},\, \omega\beta\omega^{-1} = \alpha\beta^r,\\
  & \omega\rho\omega^{-1} = \alpha^t\beta^{t(r+1)}\rho
  \rangle.
\end{align*}

\bigskip
   
As a result of the above we obtain:

\begin{theorem}\label{theorem:outE}
\begin{enumerate}[I)]
\item If $\Aut_0^1(E)=\emptyset$ we have
\begin{align*}
  \Out(E) = \langle \alpha,\,\beta,\,\rho, \mid \:&
  \alpha^2=\beta^{2s}=\rho^2=1,\\
  &\alpha\beta=\beta\alpha,\, \rho\alpha\rho=\alpha,\, \rho\beta\rho=\beta^{-1}
  \rangle,
\end{align*}
hence  $\Out(E) \cong (\zz_2\oplus\zz_{2s})\rtimes_{-1}\zz_2$.\\
\item If  $\Aut_0^1(E)\ne\emptyset$
and $\theta(1) = M_0^{\ell_0} = (M_0^\ell)^2$, 
$\det(M_0^\ell) = 1$,  $M_0^\ell = \begin{pmatrix}
  r&s\\t&r\end{pmatrix}$ and the sapphire is defined by
  $B=\begin{pmatrix}r&-t\\-s&r\end{pmatrix}$, we have 
   $\Out(E) \cong
[(\zz_2\oplus\zz_{2t})\rtimes_{-1}\zz_2]\rtimes_\omega \zz_2$ where 
\begin{align*}
   \omega\alpha\omega^{-1} = \alpha^r\beta^{st},\, \omega\beta\omega^{-1} = \alpha\beta^r,  \omega\rho\omega^{-1} = \alpha^t\beta^{t(r+1)}\rho.
\end{align*}\\
\item If   $\Aut_0^1(E)\ne\emptyset$ and $\theta(1) = -M_0^{\ell_0} = -(M_0^\ell)^2$,
$\det(M_0^\ell) = -1$, $M_0^\ell = \begin{pmatrix}
  r&s\\t&r\end{pmatrix}$ and the sapphire is defined by
  $B=\begin{pmatrix}r&-t\\s&-r\end{pmatrix}$, we have 
 a short exact sequence   

\[
\begin{tikzcd}
  1 \ar[r] & (\zz_2\oplus\zz_{2t})\rtimes_{-1}\zz_2 \ar[r] & \Out(E) \ar[r] & \zz_2 \ar[r] & 1
\end{tikzcd}
\]
and a presentation of $\Out(E)$ is given by
\begin{align*}
  \Out(E) = \langle \alpha,\,\beta,\,\rho,\,\omega \mid \:&
  \alpha^2=\beta^{2t}=\rho^2=1,\\
  & \alpha\beta=\beta\alpha,\, \rho\alpha\rho=\alpha,\, \rho\beta\rho=\beta^{-1},\\
  & \omega^2 = \beta^t\rho,\, \omega\alpha\omega^{-1} = \alpha^r\beta^{st},\, \omega\beta\omega^{-1} = \alpha\beta^r,\\
  & \omega\rho\omega^{-1} = \alpha^t\beta^{t(r+1)}\rho
  \rangle.
\end{align*}
\end{enumerate}
\end{theorem}

\section {Appendix}

The main goal of this appendix is to provide a classification of the Torus
bundles $E$ which admits a homeomorphism such that the induced
automorphism on the quotient of the fundamental group
$\pi_1(E)/(\zz\oplus\zz)\cong\zz$ is multiplication by $-1$. We also  show some results about $\Aut(E)$.

\subsection{Equations on $GL_2(\zz)$}\label{subsection:eqgl2}

We begin by recalling some results from \cite{sakuma} which are
relevant for our purpose. Given $A\in GL_2(\zz)$, $M_A$ denotes the torus bundles associated  to the homeomorphism of the torus determined by the 
matrix $A$.

\begin{lemma}\label{sak}\cite[Lemma 1.1]{sakuma}
  Let $A$ and $B$ be matrices in $GL_2(\zz)$ . Then the following
  conditions are equivalent:\\
\begin{enumerate}[(1)]
\item $M_A$ is homeomorphic to $M_B$.
\item $\pi_1(M_A)$ is isomorphic to $\pi_1(M_B)$.
\item The $\zz\langle t\rangle$-module $H_A$ is isomorphic or
  anti-isomorphic to the $\zz\langle t\rangle$-module $H_B$.
\item $A$ is conjugate to $B$ or $B^{-1}$.
\end{enumerate}
\end{lemma}

For a matrix  $A$ in $GL_2(\zz)$ we define the following sets:

\begin{definition} 
\begin{enumerate}[(1)]
\item For a matrix $A$ in $GL_2(\zz)$ let 
$C(A)=\{B\in GL_2(\zz) | BAB^{-1} =A\}$ and let
$R(A) = \{B\in GL_2(\zz) | BAB^{-1} = A^{-1}\}$,
\item $A$ is called {\it exceptional}, if one of the following conditions is satisfied:
  \begin{enumerate}[(i)]
  \item $\det(A)=1$ and $|tr(A)|\leq 2$.
  \item $\det(A)=-1$ and $tr(A)=0$.
  \end{enumerate}
\item $A$ is called Anosov if $A$ is not exceptional.
\end{enumerate}
\end{definition}

One way to provide the classification of the matrices $A$ such that
$R(A)\ne \emptyset$ is by means of Lemma 1.7, item 2,
in~\cite{sakuma}.  Unfortunately the criterion is not too practical to
be used in our present work. Once we know that for a given Torus
bundle determined by a matrix $A$ there is a $B_0\in R(A)$, then it is
easy to see that $R(A)=\{B \in GL_2(\zz) \::\: B\theta
B^{-1}=\theta^{-1}\}$ is the set of elements of the form $B_0C$ as $C$
runs over $C(A)$.

We have seen that given $\theta$ if $R(\theta)=\emptyset$ , then
$\Aut(E)=\Aut_0(E)$ and this group is given by Proposition
\ref{ident}.

So suppose that the equation $B\theta B^{-1}=\theta^{-1}$ has a
solution $B_0$.  We provide an explicit classification, up to
conjugacy class, of the matrices $\theta$ where this equation admits
at least one solution $B_0$. Then we use the result to  describe $\Aut(E)$, $\Out(E)$  which is given in  Theorems \ref{mainres},  \ref{mainres1}, respectively.

Let $A$ be an Anosov matrix.  We quote the result that $C(A)$ is
isomorphic $\zz\oplus \zz_2$. The generator of $\zz_2$ corresponds 
 to the matriz $-I_2$. Let $A_0$ be a generator of the
summand $\zz$, which is certainly a primitive element, i.e. it is not a proper power of another matrix.  There exists an
integer $\ell$ such that $A_0^{\ell}$ is either $A$ or $-A$, but we
can not guarantee the equality $A_0^{\ell}=A$. For suppose that the
equality holds and $\ell$ is even.  Take $A_1=-A_0$.  Both elements
$A_0$ and $-A_0$ to the power $\ell$ give $A$. In any case if $A$ has
infinite order then the elements of $C(A)$ are of the form $\pm M_0^k$
as $k$ runs over the integers, where $M_0\in C(A)$ is a primitive root
of $A$. So we have to consider the two possibilities.

We quote now Lemma 1.7 from \cite{sakuma}.

\begin{lemma}\label{sak1}
  If $A$ is exceptional then it is conjugate to one and only one of
  the following matrices:

\[\begin{pmatrix}
0 & -1 \\
1 & 0
\end{pmatrix},  
\pm\begin{pmatrix}
0 & -1 \\
1 & 1
\end{pmatrix},
 \pm \begin{pmatrix}
1 & n \\
0 & 1
\end{pmatrix}(n\geq 0),
  \begin{pmatrix}
1 & 0 \\
0 & -1
\end{pmatrix}, 
  \begin{pmatrix}
1 & 1 \\
0 & -1
\end{pmatrix}.
\]

Moreover we have the following:

$C\left(\begin{pmatrix}
0 & -1 \\
1 & 0
\end{pmatrix}\right)= 
\left\{ \pm \begin{pmatrix}
0 & 1 \\
1 & 0
\end{pmatrix},
\pm \begin{pmatrix}
0 & -1 \\
1 & 0
\end{pmatrix}\right\}\cong \zz_4$

$R\left(\begin{pmatrix}
0 & -1 \\
1 & 0
\end{pmatrix}\right)= 
 \begin{pmatrix}
1 & 0 \\
0 & -1
\end{pmatrix}
C\left(\begin{pmatrix}
0 & -1 \\
1 & 0
\end{pmatrix}\right)$

$C\left(\begin{pmatrix}
0 & -1 \\
1 & 1
\end{pmatrix}\right)= 
\left\{ \pm \begin{pmatrix}
1 & 0 \\
0 & 1
\end{pmatrix},
\pm \begin{pmatrix}
0 & -1 \\
1 & 1
\end{pmatrix},
 \pm \begin{pmatrix}
-1 & -1 \\
1 & 0
\end{pmatrix}\right\}\cong \zz_6$

$R\left(\begin{pmatrix}
0 & -1 \\
1 & 1
\end{pmatrix}\right)= 
 \begin{pmatrix}
1 & 0 \\
0 & -1
\end{pmatrix}
C\left(\begin{pmatrix}
0 & -1 \\
1 & 1
\end{pmatrix}\right)$

$C\left(\begin{pmatrix}
1 & n \\
0 & 1
\end{pmatrix}\right)= 
\left\{ \pm \begin{pmatrix}
1 & k \\
0 & 1
\end{pmatrix}(k\in \zz) \right\}\cong \zz\oplus\zz_2 \qquad (n\ne 0)$

$R\left(\begin{pmatrix}
1 & n \\
0 & 1
\end{pmatrix}\right)= 
 \begin{pmatrix}
1 & 0 \\
0 & -1
\end{pmatrix}
C\left(\begin{pmatrix}
1 & n \\
0 & 1
\end{pmatrix}\right) \qquad (n\ne 0)$

$C\left(\begin{pmatrix}
1 & 0 \\
0 & -1
\end{pmatrix}\right)= 
R\left(\begin{pmatrix}
1 & 0 \\
0 & -1
\end{pmatrix}\right)=
\left\{\pm \begin{pmatrix}
1 & 0 \\
0 & 1
\end{pmatrix},
 \pm \begin{pmatrix}
1 & 0 \\
0 & -1
\end{pmatrix}\right\} \cong \zz_2 \oplus \zz_2$

$C\left(\begin{pmatrix}
1 & 1 \\
0 & -1
\end{pmatrix}\right)= 
R\left(\begin{pmatrix}
1 & 1 \\
0 & -1
\end{pmatrix}\right)=
\left\{\pm \begin{pmatrix}
1 & 0 \\
0 & 1
\end{pmatrix},
 \pm \begin{pmatrix}
0 & 1 \\
0 & -1
\end{pmatrix}\right\} \cong \zz_2 \oplus \zz_2$
\end{lemma}


\begin{corollary}
Let $A=\varepsilon A_0^n$, $\varepsilon\in \{\pm 1\}$, be an Anosov
matrix where $A_0$ is a primitive root of $A$. Then $C(A)=\{\pm
A_0^i \mid i\in \zz\}\cong \zz\oplus\zz_2$.
\end{corollary}

From the results above we will show:

\begin{corollary}\label{conj}
  Suppose that $A$ has infinite order. If $B_0 A B_0^{-1}=A^{-1}$ and $B_0$ is of finite order then 
$B_0$ is conjugate to one of the following 3 matrices: 
\[
\begin{pmatrix}
0 & -1 \\
1 & 0
\end{pmatrix},  
\begin{pmatrix}
  1 & 0 \\
  0 & -1
\end{pmatrix}, 
\begin{pmatrix}
  1 & 1 \\
  0 & -1
\end{pmatrix},
\]
whuich have orders equal to $4$, $2$ and $2$, respectively.
Furthermore, $C(A)\cong\zz\oplus\zz_2$ if and only if $A$ has infinite
order.
\end{corollary}

\begin{proof}
  By the Lemma \ref{sak1} it suffices to show that $B_0$ cannot have
  order either $3$ or $6$. If $B_0^3=I_2$ it follows
  that $$A=B_0^3AB_0^{-3}=(B_0(B_0(B_0AB_0^{-1})B_0^{-1})B_0^{-1}=
  (B_0(B_0(A^{-1})B_0^{-1})B_0^{-1}=(B_0AB_0^{-1})=A^{-1}.$$ So
  $A=A^{-1}$ or $A^2=I_2$, which is a contradiction since $A$ has
  infinite order.  If $B_0^6=I_2$ it follows that $B_0$ is conjugate of
  $\begin{pmatrix} 0 & -1 \\ 1 & 1
\end{pmatrix},$   so $B_0^3$ is conjugate of the cube of 
$\begin{pmatrix}
0 & -1 \\
1 & 1
\end{pmatrix},$ 
which is $-I_2$. Therefore $B_0^3AB_0^{-3}=(-I_2) A(-I_2)=A$. But as
before, $B_0^3AB_0^{-3}= A^{-1}$ and we get a contradiction. The
``furthermore'' part follows promptly from the two Lemmas above.
\end{proof}

Observe that given any matrix, if it has order 4, then it is conjugate
to the matrix $ \begin{pmatrix} 0 & -1 \\ 1 & 0
\end{pmatrix}.$
If it has order 2, then we look at the matrix mod $2$ and then, if it is
the identity, it is conjugate to $\begin{pmatrix} 1 & 0 \\ 0 & -1
\end{pmatrix}$, otherwise it has order two and it is conjugate to 
$ \begin{pmatrix} 1 & 1 \\ 0 & -1
\end{pmatrix}$.
  
\begin{lemma}
If $A\in GL_2(\zz)$ has infinite order, $C(A)=C(-A)=C(A^k)$ for $k\ne
0$. In particular, this holds if $A$ is Anosov.
\end{lemma}
\begin{proof}
Straightforward.
\end{proof}

Using the lemma above we show the following proposition, which is
interesting in its own right.

\begin{proposition}
Let $A$ be Anosov and $B_0$ a matriz such that
$B_0AB_0^{-1}=A^{-1}$. Then $B_0^2=\pm I_2$ and $B_0\ne -I_2$.
\end{proposition}

\begin{proof}
The first observation is that $B_0$ is not Anosov. In fact, if $B_0$
is Anosov, then $B_0^2AB_0^{-2}=A  \Rightarrow A\in C(B_0^2) = C(B_0)$,
hence $B_0AB_0^{-1} = A = A^{-1}$, which cannot happen for $A$ Anosov.

Given that $B_0$ is exceptional, suppose first that $B_0$ has infinite
order, that is, it is conjugate to $\pm\begin{pmatrix}1 & n \\ 0 &
1\end{pmatrix}$ for some $n\ne 0$. So for an appropriate matrix
$A_1=\begin{pmatrix}a & b\\c & d\end{pmatrix}$, conjugate of $A$, and
for $B_1=\pm\begin{pmatrix}1 & n \\ 0 & 1\end{pmatrix}$, we have
$B_1A_1B_1^{-1}A_1 = I_2$ if, and only if,
\[
\left|
\begin{array}{l}
a^2-n^2c^2+bc+ncd=1\\
c(a-nc+d)=0\\
ab+nbc-nad-n^2cd+bd+nd^2=0\\
bc-ncd+d^2=1.
\end{array}
\right.
\]
If $c=0$, then we have $a^2=d^2=1$ and $A$ is not Anosov. Hence $c \ne
0$ and the system is equivalent to
\[
\left|
\begin{array}{l}
(a-nc)(a+nc)+bc+ncd=1\\
a-nc=-d\\
ab+nbc-nad-n^2cd+bd+nd^2=0\\
bc-ncd+d^2=1
\end{array}
\right.
\iff
\left|
\begin{array}{l}
ad-bc=-1\\
a-nc=-d\\
ab+n-n^2cd+bd+nd^2=0\\
bc-ncd+d^2=1.
\end{array}
\right.
\]
Solving the third and fourth equations for $b$ making use of equation
(2) yields to \break $b = \dfrac{1-d^2+ncd}{c} = \dfrac{-1-d^2+ncd}{c}$,
which is a contradiction.

\end{proof}


Now if we assume that $B_0$ is any of the 3 matrices of the statement
of Corollary \ref{conj} , we can find all matrices $A$ of infinite
order such that $B_0AB_0^{-1}=A^{-1}$.

\begin{proposition}\label{eqanti}
  Let $B_0AB_0^{-1}=A^{-1}$. The solutions of this equation for $A$
  are:
  \begin{enumerate}[I)]
  \item If $B_0=\begin{pmatrix} 0 & -1 \\ 1 & 0
  \end{pmatrix}$, then $A=\begin{pmatrix}
    a & b \\
    b & d
  \end{pmatrix}$,
    where $ad-b^2=1$  or  $\det(A)=1$.
  \item If $B_0 = \begin{pmatrix}
    1 & 0 \\
    0 & -1
  \end{pmatrix}$, then
    $A=\begin{pmatrix}
    a &  b\\
    c & a
  \end{pmatrix}$,
    where $a^2-bc=1$ or $\det(A)=1$.
    \item If $B_0 = \begin{pmatrix}
      1 & 1 \\
      0 & -1
    \end{pmatrix}$,   
      then either 
      
$A=\begin{pmatrix}
a &  b\\
d-a & d
\end{pmatrix}$ for   arbitrary $a,b, d$ such that  
 $\det(A)=1$ or 
      
      $A=\begin{pmatrix}
      1 &  1\\
      0 & -1
      \end{pmatrix}$ or
      $A=\begin{pmatrix}
      -1 &  -1\\
      0 & 1
      \end{pmatrix}$.
  \end{enumerate}

\end{proposition} 

\begin{proof} {Case I:}   
 $B_0=\begin{pmatrix}
0 & -1 \\
1 & 0
\end{pmatrix}.$   In this case by solving the system $B_0AB_0^{-1}=A^{-1}$ we find that  the solution for the matrices $A$ are of the form 
 $A=\begin{pmatrix}
  a & b \\
  b & d
  \end{pmatrix}$,
  where $ad-b^2=1$  or  $\det(A)=1$.

{Case II:}
$B_0 = \begin{pmatrix}
1 & 0 \\
0 & -1
\end{pmatrix}.$  In this case by solving the system $B_0AB_0^{-1}=A^{-1}$ we find that 
$A=\begin{pmatrix}
a &  b\\
c & a
\end{pmatrix}$,
where $a^2-bc=1$ or $\det(A)=1$

{Case III:}
$B_0 = \begin{pmatrix}
1 & 1 \\
0 & -1
\end{pmatrix}$. In this case by solving the system $B_0AB_0^{-1}=A^{-1}$ for $A=\begin{pmatrix}a & b \\ c & d\end{pmatrix}$ we find that $c=0$ or $c=d-a$. For $c=0$ we have the following matrices as solutions:
\[
A=\begin{pmatrix}
1 &  1\\
0 & -1
\end{pmatrix}\text{, }
A=\begin{pmatrix}
-1 &  -1\\
0 & 1
\end{pmatrix}\text{ and } A=\begin{pmatrix}
1 &  b\\
0 & 1
\end{pmatrix}\text{ for arbitrary $b$}.
\]
For $c=d-a$ we have the matrix
$A=\begin{pmatrix}
a &  b\\
d-a & d
\end{pmatrix}$. By straightforward calculation the equation $B_0AB_0^{-1}=A^{-1}$ holds if and only if  
$ad-bd+ab=1$, i.e.  if $\det(A)=1$ and $d^2-ab=1$.  So the result follows.
\end{proof}

\begin{remark}  a)The Proposition above defines three families $\mathbb{F}_1, \mathbb{F}_2, \mathbb{F}_3$ matrices in $ GL_2(\zz)$ given by the items I), II) and III),  respectively. Namely $\mathbb{F}_i$ is the family of matrices which are conjugated to one of the matrices given by the item $i)$.  
    Certainly the familes are not disjoint since 
$\begin{pmatrix}
1 &  1\\
1 & 2
\end{pmatrix}$
\noindent belong to $\mathbb{F}_1$ and $\mathbb{F}_3$. 
  
b)   We claim that  $\mathbb{F}_3$ is not contained in the union of the other two families. For let 
$A=\begin{pmatrix}
2 &  5\\
1 & 3
\end{pmatrix}$, \noindent which belongs to the family $\mathbb{F}_3$. 
\noindent The matrix $A$ cannot belong to the family  $\mathbb{F}_2$ since its  trace  $A$ is odd. We also claim that that $A$ cannot belong to the family $\mathbb{F}_1$, i.e. 
$A$ is not conjugated to a  matrix of the form     $ \begin{pmatrix}
a & b \\
b & d
\end{pmatrix}$ having determinant $1$.  This can be verified is by  directly calculation. Suppose that $A$ is conjugated to a matrix of the form

$A_1=\begin{pmatrix}
a &  b\\
b & 5-a
\end{pmatrix}$, 
\noindent  $\det(A_1)=1$ i.e. $a^2-5a+b^2+1=0$.  But we claim that the quadratic equation 
$a^2-5a+b^2+1=0$ does not admit an  integer solution. For the discriminant   of the equation 
 is $ \Delta=25-4-4b^2=21-4b^2$ and must be $\geq0$. This implies that $|b|\leq 2$ so we have three possibilities which are $|b|=0, |b|=1, |b|=2$.  But  for  $|b|=0$ we have $\Delta=21$,  for  $|b|=1$ we have $\Delta=21$, for 
  $|b|=2$ we have $\Delta=5$. So in all these 3 cases the square root of $\Delta$ is not an integer and the result follows.
  
%
%
%
%
%
%
%
%
%
%
%
%
%
%
%
%
%
%
%
%
%
%
    We do not know if the same happens for $\mathbb{F}_1$ and for  $\mathbb{F}_2$, i.e., if one of these families are contained in the union of the other two families. 
   \end{remark} 
Now we prove the major result.

\begin{theorem}\label{eqsl2}  If $A$ has infinite order and $B_0AB_0^{-1}=A^{-1}$  then $B_0$ is conjugate to one 
of the three  matrices given by Corollary \ref{conj} and $A$ is conjugate to one of the matrices $A$ given by 
Proposition \ref{eqanti}. Further, if $A$ is Anosov then $A$ is conjugate to one  of  the  matrices of the form
\begin{gather*}
A=\begin{pmatrix}
a & b \\
b & d
\end{pmatrix},
\text{ where $\det(A)=1$ and $|a+d|>2$, or }\\
A=\begin{pmatrix}
a &  b\\
c & a
\end{pmatrix}
\text{ where $\det(A)=1$ and $|2a|>2$, or }\\
A=\begin{pmatrix}
a &  b\\
d-a & d
\end{pmatrix}
\text{ where $\det(A)=1$ and  $|a+d|>2$.}
\end{gather*}
\end{theorem}
\begin{proof}  The proof follows  immediatly from Proposition \ref{eqanti}.
\end{proof}

\begin{remark}
We do not know a practical algorithm to decide
  if an arbitrary matrix is conjugate to one of the matrices of the
  families $\mathbb{F}_1, \mathbb{F}_2$ and $\mathbb{F}_3$.  In \cite{daci-peter0} examples of Anosov
  matrices $A$ where $R(A)$ is non empty were given, to illustrate
  some result in fixed point theory, where the examples are matrices
  which  belong to $\mathbb{F}_1$.

  Based on the result above, we can deduce the following consequence
  for fixed point theory of the spaces mapping torus.: It  follows from
  \cite{daci-peter0} and the Theorem \ref{eqsl2} that in order to have
  an automorphism of $(\zz\oplus\zz)\rtimes_{A} \zz$ which has Reidemeister
  finite, where $A$ is Anosov, 
   $A$ is conjugated to  one of  the matrices of the form given by Theorem \ref{eqsl2}, i.e.

  $A=\begin{pmatrix} a & b \\ b & d
  \end{pmatrix}, $
  \noindent where $ad-b^2=1$, i.e.     $\det(A)=1$,\\
 or 
 
 $A=\begin{pmatrix}
 a &  b\\
 c & a
 \end{pmatrix}, $
 \noindent where $a^2-bc=1$, i.e.  $\det(A)=1$,\\
 or       
 
 $A=\begin{pmatrix}
 a &  b\\
 d-a & d
 \end{pmatrix}, $
 \noindent where $ad-bd+ab=1$, i.e.  $\det(A)=1$.\\
\noindent We claim that in the former case we are able to construct    an automorphism with Reidemeister number finite.
In fact they are the  automorphisms which when restricted to $\zz\oplus\zz$ have matrix  
 $B_0=\begin{pmatrix}
0 & -1 \\
1 & 0
\end{pmatrix}.$ 

For the  second case   we claim that it is not possible, since the product of the matriz 
$B_0 = \begin{pmatrix}
1 & 0 \\
0 & -1

\end{pmatrix}$ 

by any matrix of the form 

$A = \begin{pmatrix}
a & b \\
b & a
\end{pmatrix}$ 
is a matrix which has Reidemeister number infinite.

For the  third  case   we also claim that it is not possible to find such examples in fixed point theory. For the product of the matriz 
$B_0 = \begin{pmatrix}
1 & 1 \\
0 & -1

\end{pmatrix}$ 

by  matrix of the form 

$A = \begin{pmatrix}
a & b \\
d-a & d
\end{pmatrix}$ 
\noindent is the  matrix 
$\begin{pmatrix}
d & b+d \\
a-d & -d
\end{pmatrix}$ 
\noindent which has Reidemeister infinite since $\det(A)=1$. The conclusion is that the type of  examples provided  in \cite{daci-peter0} 
are the only ones which have Reidemeister finite. By type we mean that the elements of the anti-diagonal are the same.  
\end{remark}

\subsection{About $\Aut_0^1(E)$}
We have seen that our main result about $\Aut(E)$ and $\Out(E)$  depends if $\Aut_0^1(E)$ is empty or not, which in turn is equivalent to say that a certain matrix has a square root or not. Here we 
show a lemma which describes for certain matrices, closely related to our study,  its square roots. Then this can be used to decide if    $\Aut_0^1(E)$ is empty or not.  \\
A practical criterion to decide if  a matrix has a square root was given by  Paulo Agozzini,  \cite{Ag}. He shows:

\begin{lemma}\label{PA}  The square roots of a matrix $A$ are either     $\pm(A+Id)/\sqrt{T+2}$ or $\pm(A-Id)/\sqrt{T-2}$ where the   matrices which are the square roots  are in $GL_2(\zz)$. Otherwise $A$ does not admite a square root. 
\end{lemma}


It follows from  the criterion given by Lemma \ref{PA} that if a matrix has a square root, then  at most one of the two numbers $T+2$, $T-2$ has integral square root. In our case we have $T=2x$ and we have the two numbers $2x+2$ and $2x-2$.  Besides one of the two numbers having  integral root square, all the enters of the matrix $2y$ and $2z$ must be divisible by the square root.  

Now we apply the lemma above  in order to describe   all matrices  of the form 
 $\begin{pmatrix}
x &  2y\\
2z & x
\end{pmatrix}, $
which admit  square root.

\begin{lemma}\label{rootl} Consider the matrizes of the form 
 $\begin{pmatrix}
2\lambda^2+1 &  2\lambda y_1\\
2\lambda z_1 & 2\lambda^2+1
\end{pmatrix}, $
\noindent where $\lambda$ runs over  $\zz$ and all $y_1, z_1$ such that $y_1z_1=\lambda^2+1$

or 

 $\begin{pmatrix}
2\lambda^2-1 &  2\lambda y_1\\
2\lambda z_1 & 2\lambda^2-1
\end{pmatrix}, $
\noindent where $\lambda$ runs over  $\zz$ and all $y_1, z_1$ such that $y_1z_1=\lambda^2-1.$

 If a matrix of the form  $\begin{pmatrix}
x &  2y\\
2z & x
\end{pmatrix}, $
admits a square root, then it is one of the matrices  above.
\end{lemma}
\begin{proof} Certainly  the matrices of the form given by Lemma  \ref{rootl} admit a square root. Now using Lemma \ref{PA} we obtain the result.
\end{proof}


\noindent{\bf Department of Mathematics-IME\\
University of S\~ao Paulo\\
Rua~do~Mat\~ao~1010\\
CEP:~05508--090, S\~ao Paulo --- SP, Brasil}\\
e-mail: dlgoncal@ime.usp.br

\vspace{1.5mm}

\noindent{\bf Department of Mathematics\\
Federal University of Santa Catarina\\
Campus Universit\'ario Trindade\\
CEP:~88040--900, Florian\'opolis --- SC, Brasil}\\
e-mail: sergio.tadao.martins@ufsc.br

\end{document}